%% file: main.tex
\documentclass[11pt,english]{article}
\usepackage[utf8]{inputenc}
\usepackage[T1]{fontenc}
\usepackage{babel}
\input{bibmath.tex}
 \author{Mohamed Slim Kammoun
 \thanks{Partially supported  by the Labex CEMPI ANR-11-LABX-0007-01.}}
\affil{Laboratoire Paul Painlevé, Université de Lille}
\providecommand{\keywords}[1]{\textbf{\textit{Keywords:}} #1}
\title{Monotonous subsequences and  the descent process of invariant random permutations}
\newtheorem{theorem}{Theorem}
\newtheorem{corollary}[theorem]{Corollary}
\newtheorem{lemma}[theorem]{Lemma}
\newtheorem{proposition}[theorem]{Proposition}
\newtheorem{conjecture}[theorem]{Conjecture}
\theoremstyle{definition}
\newtheorem{definition}[theorem]{Definition}
\usepackage{tikz}
\usetikzlibrary{automata, positioning}
\usepackage{natbib}
\begin{document}
\maketitle

\begin{abstract} 
It is known from the work of \cite*{Baik1999} that we have Tracy-Widom  fluctuations for the longest increasing subsequence of  uniform permutations. In this paper, we prove that this result holds also in the case of the Ewens distribution and more generally for a class of random permutations with distribution invariant under conjugation. Moreover, we obtain the convergence of the first components of the associated Young tableaux  to the Airy Ensemble as well as the global convergence to the Vershik-Kerov-Logan-Shepp shape. Using similar techniques, we also prove that the limiting descent process of a large class of random permutations is stationary, one-dependent and determinantal. 
\end{abstract}
\section{Introduction and statement of results}
\subsection{Monotonous subsequences}
\paragraph*{}
Let $\mathfrak{S}_n$ be the symmetric group, namely the  group of permutations of $\{1,\dots,n\}$. Given $\sigma \in \mathfrak{S}_n$, a subsequence $(\sigma(i_1),\dots,\sigma(i_k))$ is an increasing (resp. decreasing) subsequence of $\sigma$ of length $k$ if $i_1<i_2<\dots<i_k$ and $\sigma(i_1)<\dots<\sigma(i_k)$ (resp. $\sigma(i_1)>\dots>\sigma(i_k)$). We denote by $\ell(\sigma)$  (resp. $\underline{\ell}(\sigma)$) the length of the longest increasing (resp. decreasing) subsequence of $\sigma$. For example, for the permutation \begin{equation*}\sigma=\begin{pmatrix}
 1& 2 & 3 & 4 & 5 \\ 
 5& 3 & 2 & 1 & 4 
\end{pmatrix},\end{equation*} we have $\ell(\sigma)=2$ and $\underline{\ell}(\sigma)=4$. 
The study of the limiting behaviour of $\ell(\sigma_n)$ when $\sigma_n$ is a uniform random permutation is known as {Ulam's problem}: \cite{MR0129165} conjectured that the limit
\begin{equation*}
\lim_ {n\to \infty} \frac{\mathbb{E}(\ell(\sigma_n))}{\sqrt{n}}
\end{equation*}
exists. \cite{MR0480398} proved that this limit is equal to $2$. The  asymptotic fluctuations were studied by Baik, Deift and Johansson. They proved the following result:   
\begin{theorem} \citep*{Baik1999} \label{dbj}
If $\sigma_n$ is a  random permutation with the uniform distribution on $\mathfrak{S}_n$ then
\begin{equation*} 
\lim_{n \to \infty} \mathbb{P}\left(\frac{\ell(\sigma_n)-2\sqrt{n}}{n^\frac 16}\leq s\right)=F_2(s),
\end{equation*}
where  $F_2$ is  the cumulative distribution function of the  Tracy-Widom distribution.  
\end{theorem}
The Tracy-Widom distribution appears  in many problems of  random growth, integrable probability  and as the distribution of the rescaled largest eigenvalue of many models of random matrices \citep*{doi:10.1142/S2010326311300014,2012arXiv1212.3351B}. $F_2$ can be expressed as the Fredholm determinant of the Airy kernel on $L^2(s,\infty)$,  as well as in terms of the Hastings-McLeod solution of the   Painlevé  II equation \citep*{tracy1994}. Those problems are known as a part of the {Kardar-Parisi-Zhang dimension  1+1 universality class}. Apart the uniform case, \cite{Mueller2013} studied the longest increasing subsequence for Mallows distribution.
\paragraph*{}
This work's first aim is to study  the limiting behaviour of other distributions of random permutations, in particular,  to prove  a similar result to that of Baik, Deift and Johansson (Theorem \ref{dbj}). 
More precisely, we are interested  in a class of random permutations which are stable under conjugation for which we provide a  sufficient condition to obtain  the Tracy-Widom fluctuations. It includes  the  Ewens distributions and other distributions appearing in genetics, random fragmentations and coagulation processes \citep*{EWENS197287,10.2307/2984986,Kingman1,9780521867283}.
\paragraph*{} For the  remainder of this article, we denote by $(\sigma_n)_{n\geq 1}$  a sequence of random permutations with joint distribution $\mathbb{P}$ such that for all positive integer $n$,  $\sigma_n \in \mathfrak{S}_n$. We denote by  $\#(\sigma)$ the number of cycles of a permutation $\sigma$. For example, the identity of $\mathfrak{S}_n$ has $n$ cycles.  We prove  the following.
\begin{theorem} \label{the1}
Assume that the sequence of random permutations  $(\sigma_n)_{n\geq 1}$ satisfies:
\begin{itemize}
\item  For all positive integer $n$, $\sigma_n$ is stable under conjugation i.e.  $\forall \sigma , \rho \in \mathfrak{S}_n$,
\begin{equation}\tag{H1}\label{h1}
\mathbb{P}(\sigma_n=\sigma)=\mathbb{P}(\sigma_n=\rho^{-1}\sigma\rho).
\end{equation}
\item The number of cycles is such that: For all $\varepsilon>0$,
\begin{equation}\tag{H2}\label{h2}
\lim_{n\to \infty}\mathbb{P}\left(\frac{\#(\sigma_n)}{n^\frac 16 }>\varepsilon\right) =0.
\end{equation}
\end{itemize}
Then  for all  $s \in \mathbb{R}$,
\begin{equation} \label{TW} \tag{TW}
\lim_{n\to \infty} \mathbb{P}\left(\frac{\ell(\sigma_n)-2\sqrt{n}}{n^\frac 16}\leq s\right)=\lim_{n\to \infty} \mathbb{P}\left(\frac{\underline{\ell}(\sigma_n)-2\sqrt{n}}{n^\frac 16}\leq s\right)=F_2(s).
\end{equation}
\end{theorem}
The idea of the proof  we give in  Subsection \ref{proof1} is to construct a coupling   between any distribution satisfying these hypotheses and the uniform distribution in order to use Theorem \ref{dbj}. Let us  illustrate  {Theorem \ref{the1}} with the   Ewens distributions that were introduced by \cite{EWENS197287} to describe the mutation of alleles.
\begin{definition} \label{Ewens} 
Let $\theta$ be a non-negative real number. We say that a random permutation $\sigma_n$  follows the Ewens distribution with parameter $\theta$ if for all $\sigma \in \mathfrak{S}_n$,
\begin{align*}
\mathbb{P}(\sigma_n=\sigma)= \frac{\theta^{\#(\sigma)-1}}{\prod_{i=1}^{n-1}(\theta+i)}.
\end{align*}
\end{definition} 
Note that when $\theta=1$, the Ewens distribution is just the uniform distribution on $\mathfrak{S}_n$, whereas when $\theta=0$ we have the uniform distribution on permutations having a unique cycle. For general $\theta$, the Ewens distribution is clearly invariant under conjugation since it only involves the cycles' structure of $\theta$.
For our purpose, a useful property is that, if $\sigma_n$ follows the Ewens distribution with parameter $\theta>0$, then the number of cycles
$\#(\sigma_n)$ is the sum of $n$ independent Bernoulli random variables with parameters $\left\{\frac{\theta}{\theta+i}\right\}_{0\leq i \leq n-1}$. For further reading, we recommend \citep*{aldous,McCullagh2011,chafai:hal-00806514}. This already yields the following:
\begin{corollary} \label{2.1}
Let $(\theta_n)_{n\geq 1}$ be a sequence of  non-negative real numbers
such that:
\begin{equation}\tag{H'2}\label{ewcond}
\lim_{n\to \infty} \frac{\theta_n \log(n)  }{n^\frac 1 6}=0.
\end{equation}
If $\sigma_n$  follows the Ewens distribution  with parameter $\theta_n$, then we have Tracy-Widom fluctuations \eqref{TW}.
\end{corollary}
\begin{proof}
For $n\geq3$ and $\theta_n>0$, we have
\begin{align*}
\mathbb{E}(\#(\sigma_n))&=\sum_{i=0}^{n-1} \frac{\theta_n}{i+\theta_n} = 1 +\frac{\theta_n}{1+\theta_n}+\sum_{i=2}^{n-1}\frac{\theta_n}{i+\theta_n} \leq 2+ \theta_n\sum_{i=2}^{n-1}\int_{i}^{i+1} \frac{ dt}{t-1} \leq 2+\theta_n \log(n),
\end{align*}
whereas when $\theta_n=0$, we have $\#(\sigma_n)\overset{a.s}=1$. Thus, under \eqref{ewcond}, \eqref{h2} follows from Markov inequality. 
\end{proof} 
We will apply Theorem \ref{the1}  for a generalized version of the Ewens distributions in  Section \ref{diss}. We give also other applications for random virtual permutations in Subsection \ref{vpsec}. 
\paragraph*{} The proof of Theorem  \ref{dbj} uses determinantal point processes properties   obtained from  the Plancherel measure which is also the law of the shape of the Robinson-Schensted correspondence of  random uniform permutations, see \citep{kerov}.  We will study in the next subsection this correspondence in the non-uniform setting and we give  a more general result, see Theorem \ref{Airyens}.
\subsection{The Robinson–Schensted correspondence of random permutations} \label{RSKsub}
\paragraph*{} In this subsection, we study, under appropriate scaling,  the limiting shape and the limiting distribution  of the first components of the  image of a random permutation stable under conjugation by the Robinson-Schensted correspondence. 
\paragraph*{} Let $n$ be a positive integer. A Young diagram  $\lambda=\{\lambda_i\}_{i\geq1}$ of size $n$
is a partition of $n$ i.e. 
\begin{itemize}
\item $\forall i\geq 1$, $\lambda_i \in \mathbb{N}$,
\item $\forall i\geq 1$, $\lambda_{i+1}\leq \lambda_i$,
\item $\sum_{i=1}^\infty \lambda_i=n$.
\end{itemize}

We can represent a Young diagram by boxes of size $1\times1$ such that the row  $i$ contains exactly $\lambda_i$ boxes.  For example, if $\lambda=(4,2,1,\underline{0})$, we have the diagram
$$\yng(4,2,1),$$ where $\underline{0}=(0)_{i\geq 1}$.   Let $\mathbb{Y}_n$ be the set of Young diagrams of size $n$. For example, 
$$\mathbb{Y}_4=\{(4,\underline{0}),\,(3,1,\underline{0}),\,(2,2,\underline{0}),\,(2,1,1,\underline{0}),\,(1,1,1,1,\underline{0})\}=\left\{\,\yng(4)\,,\quad\yng(3,1)\,,\quad\yng(2,2)\,,\quad\yng(2,1,1)\,,\quad\yng(1,1,1,1)\,\right\}.$$
\paragraph*{}
In the sequel of this paper, for a young diagram $\lambda$, we denote by $\lambda'$ its conjugate  defined by ${\lambda'=(\lambda'_i)_{i\geq 1}}$ where $\lambda'_i:=|\{j;\lambda_j\geq i\}|$. For example, if $\lambda=(4,2,1,\underline{0})$, $\lambda'=(3,2,1,1,\underline{0})$.
\paragraph*{}We will use  the well-known application on the symmetric group $\mathfrak{S}_n$ with values in $\mathbb{Y}_n$ known as the shape of the image of a permutation $\sigma$ by the Robinson–Schensted correspondence \citep*{RSKR,MR0121305} or the Robinson–Schensted–Knuth correspondence \citep*{RSKK}. We denote it by $$\lambda(\sigma)=\{\lambda_i(\sigma)\}_{i\geq 1}.$$ We will not include here algorithmic details. For further reading, we recommend \citep*[Chapter 3]{Sagan2001}. For our purpose,  a useful property of this transform is that 
\begin{equation} \label{defrs}
\lambda_1(\sigma)=\ell(\sigma), \quad \lambda'_1(\sigma)=\underline{\ell}(\sigma). 
\end{equation}
\paragraph*{}
When $\sigma_n$ follows the uniform law, the distribution of $\lambda(\sigma_n)$ on $\mathbb{Y}_n$  is known as the Plancherel measure. In this case, after appropriate scaling, $\lambda(\sigma_n)$ converges at the edge to the Airy ensemble. For the definition of the Airy ensemble, which is the determinantal point process associated with the Airy kernel, see for example \citep*{tracy1994}.
\paragraph*{} In the remainder of this paper, we denote by $F_{2,k}(s_1,s_2,\dots,s_k):=\mathbb{P}(\forall i\leq k,\;\xi_i\leq s_i)$ the cumulative distribution of the top right $k$ particles of the Airy ensemble $(\xi_i)_{i\geq 1}$.
\begin{theorem}\citep*[Theorem 5]{Borodin2000}\cite[Theorem 1.4]{10.2307/2661375} \label{BOOJ}
Assume that $\sigma_n$ follows the uniform distribution on $\mathfrak{S}_n$. Then for all real numbers $s_1,s_2,\dots,s_k$,
\begin{align*}
\lim_{n\to \infty}\mathbb{P}\left(\forall i\leq k, \;\frac{\lambda_i(\sigma_n)-2\sqrt{n}}{n^\frac{1}{6}}\leq s_i\right)
= F_{2,k}(s_1,s_2,\dots,s_k).
\end{align*}
\end{theorem}
\paragraph*{} For distributions satisfying the  same assumptions  as in  Theorem \ref{the1}, we have the same asymptotic as in the uniform setting at the edge.
 \begin{theorem}\label{Airyens}
Assume that the sequence of random permutations  $(\sigma_n)_{n\geq 1}$ satisfies \eqref{h1} and \eqref{h2}. Then for all positive integer $k$, for all real numbers $s_1,s_2,\dots,s_k$,
\begin{align}\tag{Ai} \label{TW2}
\lim_{n\to \infty}\mathbb{P}\left(\forall i\leq k,\frac{\lambda_i(\sigma_n)-2\sqrt{n}}{n^\frac{1}{6}}\leq s_i\right)=\lim_{n\to \infty}\mathbb{P}\left(\forall i\leq k,\frac{\lambda'_i(\sigma_n)-2\sqrt{n}}{n^\frac{1}{6}}\leq s_i\right)
= F_{2,k}(s_1,s_2,\dots,s_k).
\end{align}
\end{theorem}
  Clearly, the convergence  \eqref{TW2} holds for the Ewens distributions under the hypothesis \eqref{ewcond}.
\paragraph*{}
Using \eqref{defrs}, Theorem \ref{the1} is a direct application of this theorem for $k=1$.
 The proof we provide in Subsection \ref{RSKPROOF} is a generalization of the proof of Theorem \ref{the1}. We give separate proofs of Theorem \ref{the1} and Theorem \ref{Airyens} because the  proof of Theorem \ref{the1} is  simpler and does not require any knowledge of the representations of the symmetric group. Moreover, we believe that understanding the  proof of Theorem~\ref{the1} is  helpful to understand the main idea of the  proof of Theorem~\ref{Airyens}.      
\paragraph*{}
The typical shape  under the Plancherel measure  was studied separately by \cite{LOGAN1977206} and \cite{MR0480398}. Stronger results are proved by \cite{Vershik1985}.  In 1993, Kerov studied the limiting fluctuations  but did not publish his results. See \citep*{10.1007/978-94-010-0524-1_3} for further  details.
 Let $L_{\lambda(\sigma)}$ be the height function of $\lambda(\sigma)$   rotated by $\frac{3\pi}{4}$  and extended by the function $x\mapsto |x|$ to obtain a function defined on $\mathbb{R}$. For example,  if $\lambda(\sigma)=(7,5,2,1,1,\underline{0})$ the associated function $L_{\lambda(\sigma)}$ is represented by Figure \ref{figL}. 
\begin{figure}[ht]
\centering
\begin{tikzpicture}    [/pgfplots/y=0.5cm, /pgfplots/x=0.5cm]
      \begin{axis}[
    axis x line=center,
    axis y line=center,
    xmin=0, xmax=10,
    ymin=0, ymax=10, clip=false,
    ytick={0},
	xtick={0},
    minor xtick={0,1,2,3,3,4,5,6,7,8,9},
    minor ytick={0,1,2,3,3,4,5,6,7,8,9},
    grid=both,
    legend pos=north west,
    ymajorgrids=false,
    xmajorgrids=false, anchor=origin,
    grid style=dashed    , rotate around={45:(rel axis cs:0,0)}
,
]

\addplot[
    color=blue,
        line width=3pt,
    ]
    coordinates {
    (0,10)(0,7)(1,7)(1,5)(2,5)(2,2)(3,2)(3,1)(5,1)(5,0)(10,0)
    };
 
\end{axis}
\begin{axis}[
	axis x line=center,
    axis y line=center,
    xmin=-7.07, xmax=7.07,
    ymin=0, ymax=8, anchor=origin, clip=false,
    xtick={-7,-6,-5,-4,-3,-2,-1,0,1,2,3,4,5,6,7},
    ytick={0,1,2,3,3,4,5,6,7,8},
    legend pos=north west,
    ymajorgrids=false,
    xmajorgrids=false,rotate around={0:(rel axis cs:0,0)},
    grid style=dashed];
\end{axis}
    \end{tikzpicture}
    \caption{ $L_{(7,5,2,1,1,\underline{0})}$}
     \label{figL}
\end{figure}
For the Plancherel measure we have the following result. 

\begin{theorem}\cite[Theorem 4]{Vershik1985} \label{BOOJ2}
Assume that $\sigma_n$ follows the uniform distribution.
Then for all $\varepsilon>0$,
\begin{align*}
\lim_{n\to \infty} \mathbb{P}\left(\sup_{s\in \mathbb{R}} \left|\frac{1}{\sqrt{2n}}L_{\lambda(\sigma_n)}\left({s}{\sqrt{2n}}\right)-\Omega(s)\right|<\varepsilon\right) =1,
\end{align*}
where 
\begin{align*}
\Omega(s):=\begin{cases}
\frac{2}{\pi}(s\arcsin({s})+\sqrt{1-s^2}) & \text{ if } |s|<1 \\ 
|s| & \text{ if } |s|\geq 1 
\end{cases}.
\end{align*}
\end{theorem}
Under weaker conditions than those of Theorem \ref{Airyens}, we show a similar result. For the remainder of this paper, we will refer to this limiting shape as  the Vershik-Kerov-Logan-Shepp shape.
This convergence is closely related to the Wigner's semi-circular law. For further details, one can see \citep{ss1,ss3,ss2}.
\begin{theorem}\label{VCthm} 

\begin{figure}[ht]
    \centering
    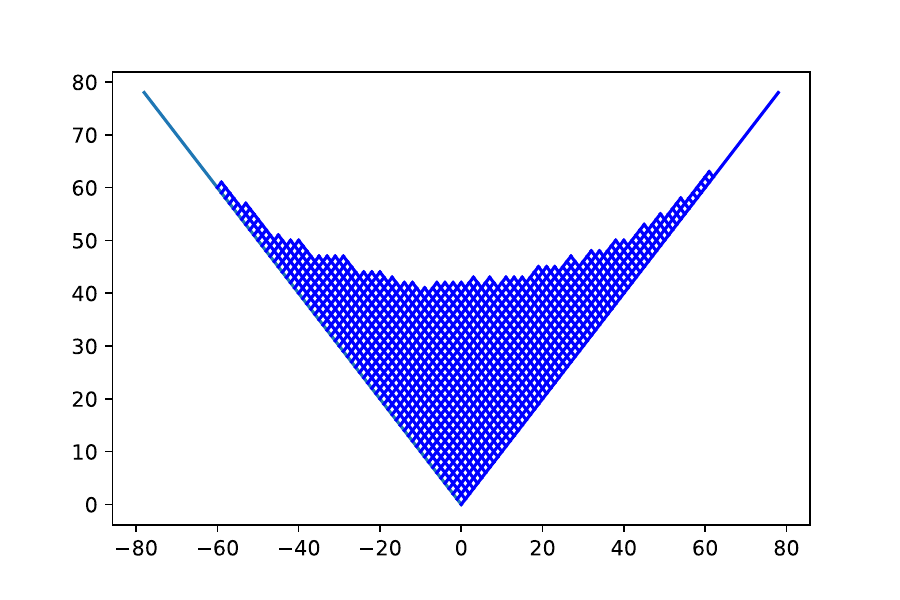
    
    \caption{Illustration of the Vershik-Kerov-Logan-Shepp convergence}
\end{figure}
Assume that the sequence of random permutations  $(\sigma_n)_{n\geq 1}$ satisfies \eqref{h1} and that
for all $\varepsilon>0$,
\begin{equation}\tag{H3}\label{H4}
\lim_{n\to \infty}\mathbb{P}\left(\frac{\#(\sigma_n)}{{n} }>\varepsilon\right) =0.
\end{equation}
Then for all $\varepsilon>0$,
\begin{align}\tag{VKLS}\label{VC}
\lim_{n\to \infty} \mathbb{P}\left(\sup_{s\in \mathbb{R}} \left|\frac{1}{\sqrt{2n}}L_{\lambda(\sigma_n)}\left({s}{\sqrt{2n}}\right)-\Omega(s)\right|<\varepsilon\right) =1.
\end{align}

\end{theorem}
We will prove this result in Subsection \ref{RSKPROOF} using the same coupling as in the proof of Theorem \ref{the1}. 
\subsection{The descent process}
\paragraph*{}
Let $n$ be a positive integer and $\sigma \in \mathfrak{S}_n$. We define  
\begin{align}
{D}(\sigma):=\{ i\in \{1,\,\dots,\,n-1\};\;\sigma(i+1)<\sigma(i)\}.
\end{align}
For example,  $$\text{for }\sigma=\begin{pmatrix}
 1& 2 & 3 & 4 & 5 \\ 
 5& 3 & 1 & 4 & 2
\end{pmatrix}, \quad D(\sigma)=\{1,\,2,\,4\}.$$
\paragraph*{} When $\sigma$ is random, $D(\sigma)$ is known as the descent process.

\begin{theorem} 

(\citep*[Theorem 5.1]{MR2721041})
\label{borodin2}
Assume that $\sigma_n$  follows the uniform distribution on $\mathfrak{S}_n$. Then for all  ${A \subset \{1,2,\dots,n-1\}}$, 
\begin{equation*}\mathbb{P}(A \subset D(\sigma_n))=\det([k_{0}(j-i)]_{i,j \in A}),
\end{equation*}
where, 
\begin{align*}
\sum_{i\in \mathbb{Z}}k_0(i)z^i=\frac{1}{1-e^z}.
\end{align*}
\end{theorem}
We say  that the descent process is determinantal with kernel $K_0(i,j):=k_0(j-i)$.
Determinantal point processes were  introduced by \cite{10.2307/1425855} to describe fermions in quantum mechanics. For further reading we refer for example to \citep{BDPP}.
\paragraph*{} In the non-uniform setting, the descent process is already studied for the Mallow's law with  Kendall tau metric: it is also determinantal  with   different kernels. See \citep*[Proposition 5.2]{MR2721041}. Using similar techniques as in the previous subsections, we  show that for a large class of random permutations, the limiting descent process is determinantal with the same kernel as the uniform setting.
\begin{theorem}\label{thm2}
Assume that the sequence of random permutations  $(\sigma_n)_{n\geq 1}$ satisfies \eqref{h1} and
\begin{equation}\tag{H4}\label{h3}
\lim_{n\to \infty} \mathbb{P}(\sigma_n(1)=1)= 0.
\end{equation}
Then  for all finite set  $A \subset \mathbb{N}^*:=\{1,2,\dots\}$, 
\begin{equation} \label{main} \tag{DPP}
\lim_{n\to \infty} \mathbb{P}(A \subset D(\sigma_n))=\det([k_0(j-i)]_{i,j \in A}).
\end{equation}
\end{theorem}
We  will prove this result in Subsection \ref{descpr} but before that  let us illustrate it by the Ewens distributions (see Definition~\ref{Ewens}).
\begin{corollary}\label{4.1}  Let $(\theta_n)_{n \geq 1}$ be a sequence of non-negative real numbers. Assume that $\sigma_n$  follows the Ewens distribution with parameter $\theta_n$. If \begin{align*} 
\lim_{n \to \infty}  \frac{\theta_n}{n}=0.
\end{align*}
Then the limiting descent process is determinantal with kernel $K_0$ \eqref{main}. \end{corollary}
\begin{proof} 
Using the Chinese restaurant process interpretation of the Ewens measures, see for example \citep*[Part II~Section 11]{aldous},  we have 
\begin{align*}
\mathbb{P}(\sigma_n(n)=n)=\frac{\theta_n}{\theta_n+n-1}.
\end{align*}
By the stability under conjugation, 
\begin{align*}
\lim_{n\to \infty} \mathbb{P}(\sigma_n(1)=1)=\lim_{n\to \infty} \mathbb{P}(\sigma_n(n)=n)=\lim_{n\to \infty}\frac{\theta_n}{\theta_n+n-1}\leq \lim_{n\to \infty} \frac{\theta_n}{n-1}= 0.
\end{align*}
We can now conclude using Theorem \ref{thm2}.
\end{proof}
When $\theta_n=0$ (the uniform measure on permutations having a unique cycle), we have a stronger result. For all positive integers $n$ and $m$ such that $m\geq n+2$,  for all $A  \subset \{1,\dots,n\}$, 
\begin{equation*} 
\mathbb{P}(A \subset D(\sigma_m))=\det([k_0(j-i)]_{i,j \in A}).
\end{equation*}
In other terms, in this case,  the restriction of the descent process of $\sigma_{n+2}$ to $\{1,2,\dots,n\}$  is determinantal  with kernel $K_0$. This result is a direct consequence of the main result of \cite{ELIZALDE2011688}. 

\subsection{Virtual permutations }
\label{vpsec}
\paragraph*{} We give in this subsection another application of previous theorems. Virtual permutations  are introduced by \cite*{MR1218259} as the projective limit of $\mathfrak{S}_n$. We are interested in this article only  in random virtual permutations  stable under conjugation also known as central measures as defined and  totally characterized   by \cite{TSILEVICH1998StationaryMO}. Those measures are the counterpart for random permutations of the Kingman exchangeable random partitions \citep*{10.2307/2984986,Kingman1}.
\paragraph*{}
Let $n$ be a positive integer and $\pi_n$ be the projection of $\mathfrak{S}_{n+1}$ on $\mathfrak{S}_n$ obtained by removing $n+1$ from the cycles' structure of the permutation. For example, $$\pi_3((1,\,3)\,(2,\,4))=\pi_3((1,\,4,\,3)\,(2))=\pi_3((1,\,3)\,(2)\,(4))=(1,\,3)\,(2).$$ 
We define the space of virtual permutations $\mathfrak{S}^\infty$ as the projective limit of $\mathfrak{S}_n$ as $n$ goes to infinity:
\begin{equation*}
\mathfrak{S}^\infty:=\{ (\hat\sigma_n)_{n\geq 1};\;\forall n\geq 1,\; \pi_n(\hat\sigma_{n+1})=\hat\sigma_n\}=\lim_{\longleftarrow }\mathfrak{S}_n.
\end{equation*} 
\paragraph{}Therefore, a  random  virtual permutation  is a sequence  $(\sigma_n)_{n \geq 1}$ of random permutations such that ${\pi_n(\sigma_{n+1})\overset{a.s}=\sigma_n}$.
We say that it is stable under conjugation if for all positive integer $n$, $\sigma_n$ is stable under conjugation. In this case, the number of cycles can be expressed in terms of probabilities of fixed points.
\begin{corollary}\label{2.3} 
Let $(\sigma_n)_{n\geq 1}$ be  a random virtual permutation stable under conjugation. Assume that 
\begin{align} \label{vpeq52} \tag{H'4}
\lim_{n\to \infty} \mathbb{P}(\sigma_n(1)=1)=0.
\end{align}
Then we have the  Vershik-Kerov-Logan-Shepp limiting shape \eqref{VC}. Moreover, if 
\begin{align} \label{vpeq5} \tag{H''2}
\mathbb{P}(\sigma_n(1)=1)=o\left({n^{-\frac{5}{6}}}\right).
\end{align}
Then we have Tracy-Widom fluctuations \eqref{TW} and the convergence at the edge to the Airy ensemble \eqref{TW2}.
\end{corollary}
 \begin{proof}
 By construction,  for all random virtual permutation $(\sigma_n)_{n\geq 1}$ and for all positive integer $n$,  
 \begin{equation*}
\#(\sigma_n)= \#(\pi_n(\sigma_{n+1}))=\#(\sigma_{n+1})-\mathbbm{1}_{\sigma_{n+1}(n+1)=n+1}.\end{equation*}
Consequently, 
\begin{align*} \mathbb{E}(\#(\sigma_{n}))=\sum_{i=1}^n \mathbb{P}(\sigma_i(i)=i)=\sum_{i=1}^n \mathbb{P}(\sigma_i(1)=1).\end{align*}
Moreover, under the hypothesis \eqref{vpeq5} we have
\begin{equation*}
\sum_{i=1}^n \mathbb{P}(\sigma_i(1)=1) =o(n^\frac{1}{6}).
\end{equation*}
 We can then conclude using Theorem \ref{Airyens}. Similarly,  using the hypothesis \eqref{vpeq52} we obtain:
\begin{equation*}
\sum_{i=1}^n \mathbb{P}(\sigma_i(1)=1) =o(n).
\end{equation*}  
We can then conclude using Theorem \ref{VCthm}.
 \end{proof}

According to \citep*[Section 2]{TSILEVICH1998StationaryMO} there exists a one-to-one correspondence between  
the set of probability distributions on $\mathfrak{S}^\infty$ stable under conjugation and the set of probability distributions on  \begin{equation*}
\Sigma:=\left\{(x_i)_{i\geq 1};\, x_1 \geq x_2 \geq  \dots\geq0,\; \sum_i {x_i}\leq 1 \right\}.
\end{equation*}
Let $0\leq a\leq 1$. We denote
 \begin{equation*}
\Sigma_a:=\left\{(x_i)_{i\geq 1};\, x_1 \geq x_2 \geq  \dots\geq0, \; \sum_i {x_i}= a \right\}.
\end{equation*}
Let $\nu$ be a probability measure on $\Sigma$. We denote by $(\sigma^{\nu}_n)_{n\geq 1}$ a random virtual permutation stable under conjugation such that the associated distribution on $\Sigma$ is $\nu$. 
We will study this correspondence in three parts:
\begin{itemize}
\item Let $x=(x_i)_{i\geq 1} \in \Sigma_1$.  If $\nu=\delta_x$, then for all positive integer $n$, for all $\sigma\in \mathfrak{S}_n$,
\begin{align} \label{deferg}
f(n,x,\sigma):=\mathbb{P}(\sigma_n^{\delta_x}=\sigma)= \prod_{j\geq 1}\frac{r_j!}{((j-1)!)^{r_j}}\sum_{m}\prod_{i\geq1}x_i^{m_i}.
\end{align}
Here, $r_j$ is the number of cycles of length $j$ of $\sigma$ and the sum is over all sequences of non-negative integers $m=(m_i)_{i\geq1}$ such that $\forall j\geq 1,|\{i;m_i=j\}|=r_j$. 
For more details, see \citep*[Section 2]{TSILEVICH1998StationaryMO}. 
\begin{corollary} \label{2.4}
If  $x_n=o(n^{-\alpha})$ with  $\alpha>6$,  then we have Tracy-Widom fluctuations \eqref{TW} and  the convergence at the edge to the Airy ensemble \eqref{TW2}.
\end{corollary}
 \begin{corollary} \label{2.4p} If  $x_n=o(n^{-\alpha})$ with  $\alpha>1$, then we have the  Vershik-Kerov-Logan-Shepp limiting shape \eqref{VC}.
\end{corollary}

We give a proof of Corollary \ref {2.4} and Corollary \ref{2.4p} in Subsection~\ref{appli}. A trivial application of these corollaries is when $x_i=\delta_1{(i)}$. In this case, $\sigma^{\delta_x}_n$ follows the Ewens distribution with parameter $\theta=0$.
\item If $\nu(\Sigma_1)=1$, $\nu$ is  called a 1-measure. In this case,  the distribution of $(\sigma^\nu_n)_{n \geq 1} $ is a mixture of the previous  distributions i.e.
for all positive integer $n$, for all $\sigma\in \mathfrak{S}_n$,
\begin{align}\label{onem} 
\mathbb{P}(\sigma^\nu_n=\sigma)=\int_{x\in\Sigma_1}f(n,x,\sigma)d\nu(x).
\end{align}

\begin{corollary} \label{2.5p}
Assume that $\nu$ is a 1-measure and 
\begin{align*}
\int_{x\in \Sigma_1}\sum_{i=1}^\infty \left(1-(1-x_i)^n\right) d\nu(x) =o\left(n^\frac{1}{6}\right),
\end{align*}
then  we have Tracy-Widom fluctuations \eqref{TW} and the convergence at the edge to the Airy ensemble \eqref{TW2}.
\end{corollary}

\begin{corollary} \label{2.5}
Assume that $\nu$ is a 1-measure. We have then the  Vershik-Kerov-Logan-Shepp limiting shape  \eqref{VC}. \end{corollary}

\paragraph*{}
We will prove Corollary \ref{2.5p} and Corollary \ref{2.5} in Subsection \ref{appli}.
To explain the relation with the  Ewens distributions, we need first to introduce the Poisson-Dirichlet distributions. Let $\theta>0$ and let ${1\geq x_1 \geq x_2 \geq\dots\geq 0}$ be a Poisson point  process on $(0,1]$ with intensity $\lambda(t)=\frac{\theta\exp(-t)}{t}$.
We define the random variable $S:=\sum_{i\geq 1} x_i$. It is proved that the sum $S$ is almost surely finite.
We can find a  proof for example in  \citep*{Holst01thepoisson-dirichlet}. The point process  $\hat{x}:=\left(\frac{x_i}{S}\right)_{i\geq1}$ defines a measure on $\Sigma_1$ known as the Poisson-Dirichlet distribution with parameter $\theta$. It was introduced by \cite*{10.2307/2984986} and it is a useful tool to study some problems of combinatorics, analytic number theory,
statistics and population genetics. See \citep*{0898711665,Donnelly1993,arratia,9780821898543}.
\paragraph*{}
The Poisson-Dirichlet distribution with parameter $\theta>0$ represents also the limiting distribution of normalized cycles' lengths of the Ewens distribution with the same parameter, see \citep*{arratia}. As a consequence, using the description of these measures in \citep*[Section 2]{TSILEVICH1998StationaryMO}, if  $\nu$ follows the Poisson-Dirichlet distribution with parameter $\theta$,  $\sigma^\nu_n$  follows the  Ewens measure with same parameter $\theta$. In this case, the hypotheses of Corollaries \ref{2.5p} and \ref{2.5} are satisfied.
\item In the general case, the correspondence is given by the formula: 
\begin{align*}
\mathbb{P}(\sigma^\nu_n=\sigma)=\int_{x\in\Sigma}f(n,x,\sigma)d\nu(x),
\end{align*}
where
\begin{align}\label{defdef}
f(n,x,\sigma):=\begin{cases} \prod_{j\geq 1}\frac{r_j!}{((j-1)!)^{r_j}}\sum_{m}\prod_{i\geq1}x_i^{m_i}
 & \text{ if } \sum_{i=1}^\infty(x_i)=1
 \\ \sum_{j=0}^{l} \binom{l}{j} x_0^j(1-x_0)^{n-j}  f(n-j,\,y,\,\sigma^j)
 & \text{ if } 0<\sum_{i=1}^\infty(x_i)<1
 \\ \mathbbm{1}_{\sigma=Id_n}  & \text{ if } \sum_{i=1}^\infty(x_i)=0
\end{cases}.
\end{align}
Here, $r_j$ is the number of cycles of length $j$ of $\sigma$ and the sum is over all sequences of non-negative integers $m=(m_i)_{i\geq1}$ such that $\forall j\geq 1,|\{i;m_i=j\}|=r_j$, $y:=\frac{x}{\sum_i x_i}$, $x_0:=1-\sum_{i=1}^\infty x_i$, $l$ is the number of fixed points of $\sigma$, $\sigma^j$ is the permutation obtained by removing $j$ fixed points of $\sigma$ and $Id_n$ is the identity of $\mathfrak{S}_n$.
For more details, we recommend \citep*[Section 2]{TSILEVICH1998StationaryMO}.
\paragraph*{} In the general case, we do not expect the Tracy-Widom fluctuations neither for $\ell$ nor for $\underline{\ell}$ (see Section~\ref{diss}). We limit then our study to the case where there exists $0<x_0<1$ such that $\nu(\Sigma_{1-x_0})=1$.
Unlike all  previous examples when $\ell(\sigma_n)$ and $\underline\ell(\sigma_n)$ have the same asymptotic fluctuations, in this case, the expected length of the longest increasing subsequence is larger than $(1-x_0)n$ and   we will show that there exist some cases   where   the expected length of the longest decreasing  subsequence is asymptotically proportional to $\sqrt{n}$ with Tracy-Widom fluctuations.   
\begin{corollary} \label{2.6}
Let $0<x_0<1$ and $\nu$  be a probability measure on $\Sigma$ satisfying $\nu\left(\Sigma_{1-x_0}\right)=1$. Let $\hat\nu$ be the 1-measure such that $d\hat\nu(x)=d\nu\left(\frac{x}{1-x_0}\right)$. If there exists a positive integer $k$ such that  for all real numbers $s_1,s_2,\dots,s_k$, 
\begin{equation*} 
\lim_{n\to \infty} \mathbb{P}\left(\forall 1\leq i \leq k,\;\frac{\lambda'_i(\sigma^{\hat{\nu}}_n)-2\sqrt{n}}{n^\frac 16}\leq s_i\right)=F_{2,k}(s_1,\,\dots,\,s_k),
\end{equation*}
then for all real numbers $s_1,s_2,\dots,s_k$, 
\begin{equation*} 
\lim_{n\to \infty} \mathbb{P}\left(\forall 1\leq i \leq k,\;\frac{\lambda'_i(\sigma^\nu_n)-2\sqrt{(1-x_0)n}}{((1-x_0)n)^\frac 16}\leq s_i\right)=F_{2,k}(s_1,\,\dots,\,s_k).
\end{equation*}
In particular, for all real $s$,
\begin{equation*} 
\lim_{n\to \infty} \mathbb{P}\left(\frac{\underline\ell(\sigma^\nu_n)-2\sqrt{(1-x_0)n}}{((1-x_0)n)^\frac 16}\leq s\right)=F_2(s).
\end{equation*}
\end{corollary}
This corollary is a direct application of Proposition \ref{them11}.
Here are some examples of measures $\nu$ that meet the assumptions of the previous corollary:   \begin{itemize}
\item  When $\nu=\delta_{x}$ and $x_i=o(\frac{1}{i^{6+\varepsilon}})$.
\item  When $d\nu(x)= dPD(\beta)(\frac{x}{\alpha})$, $\beta \geq 0$, $0<\alpha\leq1$  and $PD(\beta)$ is Poisson-Dirichlet distribution with parameter $\beta$. 
\end{itemize}
In fact:
\begin{itemize}
\item If $\nu=\delta_x$ and $x_i=o(\frac{1}{i^{6+\varepsilon}})$, then $\hat\nu=\delta_\frac{x}{\sum_{i\geq1}{x_i}}$   satisfies hypotheses of Corollary \ref{2.4}.
\item If $d\nu(x)= dPD(\beta)(\frac{x}{\alpha})$,  then $d\hat \nu(x)= dPD(\beta)(x)$ and $\hat\sigma_n$ follows the Ewens distribution with parameter $\beta$. We can then conclude using Corollary \ref{Airyens}. 
\end{itemize}
For the descent process, we have the following result:
\begin{theorem} \label{5}
 If there exists $0\leq x_0\leq 1$ such that $\nu(\Sigma_{1-x_0})=1$, then  for all finite set  $A \subset \mathbb{N}^*$, 
\begin{equation*} 
\lim_{n\to \infty} \mathbb{P}(A \subset D\left(\sigma^\nu_n)\right)=\det([k_{x_0}(j-i)]_{i,j \in A}),
\end{equation*}
with 
\begin{align}\label{keneldescvirt}
\sum_{l\in \mathbb{Z}} k_{x_0}(l)z^l=\frac{1}{1-(1+x_0z )e^{(1-x_0)z}}=\frac{-1}{z+\sum_{l=1}^\infty\hat{a}_l(x_0)z^{l+1}},
\end{align}
where 
\begin{align}\label{eqq7}
\hat{a}_l(x_0):=\frac{(1-x_0)^{l+1}}{(l+1)!}+\frac{x_0(1-x_0)^{l}}{l!}.
\end{align}
\end{theorem}
The proof of this result we suggest  in Subsection \ref{descpr} consists in studying in a first step the case where the corresponding measure $\nu$ is concentrated on $\Sigma_1$. We prove that  the limiting point process is determinantal with kernel $(i,j)\mapsto k_{0}(j-i)$. In a second step, we prove that the kernel depends only on  $\sum_{i \geq 1} x_i$.
\paragraph*{}
Theorem \ref{5} implies that 
for a general random virtual permutation stable under conjugation, we have the following result. 
\begin{corollary} \label{gcase} For any probability measure $\nu$ on $\Sigma$,
\begin{equation} \label{limvir} 
\lim_{n\to \infty} \mathbb{P}(A \subset D\left(\sigma^\nu_n)\right)=\int_\Sigma \det\left(\left[k_{1-\sum_i{x_i}}(j-i)\right]_{i,j \in A}\right)d\nu(x).
\end{equation}
\end{corollary}
For the total number of descents we have
\begin{proposition} \label{1.1} For any probability measure $\nu$ on $\Sigma$,
 \begin{equation*}
\lim_{n\to \infty} \frac{\mathbb{E}(|D(\sigma^\nu_n)|)}{n}=\frac{1}{2}\left(1-\int_\Sigma \left(1-\sum_i{x_i}\right)^2 d\nu(x)\right).
\end{equation*}
\end{proposition}
We will prove Corollary \ref{gcase} and Proposition \ref{1.1} in Subsection \ref{descpr}.
\end{itemize}
\section*{Acknowledgements}
\paragraph*{}
The author would like to acknowledge many extremely useful conversations
with   Adrien Hardy and Mylène Maïda, their supervision of this work and  their great help to elaborate  and to ameliorate the coherence of this paper. This work is partially supported by Labex CEMPI (ANR-11-LABX-0007-01).
\section{Further discussion}
\label{diss} 
\paragraph*{} In  previous subsections, except for Corollary \ref{2.1}, the applications are for virtual permutations, but with the same logic, we can prove a similar result  as Corollary \ref{2.6} for some  permutations non compatible with projections.
\begin{proposition} \label{them11}
Let $(\mathbb{P}_n)_{n\geq1}$ be a sequence of probability measures  stable under conjugation. Assume that there exists a positive integer $k$ such that  for all real numbers $s_1,s_2,\dots,s_{k}$,
\begin{equation} \label{convrh} \tag{H5}
\lim_{n\to \infty} \mathbb{P}_n\left(\left\{\sigma \in \mathfrak{S}_n,\, \forall  1 \leq i \leq k,\,  \frac{\lambda'_i(\sigma)-2\sqrt{n}}{n^\frac 16}\leq s_i\right\}\right)=F_{2,k}(s_1,\dots,s_{k}).
\end{equation}
Let $0\leq x_0< 1$ and $(\sigma_n)_{n\geq 1}$ be a sequence of random permutations such that for all positive integer $n$, for all $\sigma \in \mathfrak{S}_n$,
\begin{align}\label{def2.1}
\mathbb{P}(\sigma_n=\sigma):=\sum_{j=0}^{l} \binom{l}{j} x_0^j(1-x_0)^{n-j}  \mathbb{P}_{n-j}(\sigma^j),
\end{align}
where $l$ is the  number of fixed points of $\sigma$ and $\sigma^j$ is the permutation obtained by removing $j$ fixed points of $\sigma$.
Then  for all real numbers $s_1,s_2,\dots,s_{k}$,
\begin{equation*} 
\lim_{n\to \infty} \mathbb{P}\left(\forall  1 \leq i \leq k,\, \frac{\lambda'_i(\sigma_n)-2\sqrt{(1-x_0)n}}{((1-x_0)n)^\frac 16}\leq s_i\right)=F_{2,k}(s_1,\dots,s_{k}).
\end{equation*}
\end{proposition}
We prove this result in Subsection \ref{appli}. 
  An interpretation of the random permutation defined by equation \eqref{def2.1} is the following. Let $n$ be a positive integer. We construct  a subset $A$ of $\{1,2,\dots,n\}$ as follows: for every $1\leq i \leq n$, with probability $x_0$, $i \in A$  independently from other points. The points of $A$ are then fixed points of $\sigma_n$. After that,  we permute the elements of    $\{1,2,\dots,n\}\setminus A$ according to the probability distribution $\mathbb{P}_{n-|A|}$. In particular, $A$ is a subset of all   fixed points of $\sigma_n$.
\paragraph*{} As a consequence, recalling \eqref{defdef}, if there exists $0<x_0<1$ such that $\nu\left(\Sigma_{1-x_0}\right)=1$, then  the number of fixed points of $\sigma^\nu_n$ is larger than a binomial random variable with parameters $x_0$ and $n$. Consequently, 
\begin{align*}
\mathbb{E}(\ell(\sigma^\nu_n)) \geq n x_0. 
\end{align*}
In this case, we conjecture that the fluctuations are Gaussian. 
\begin{conjecture}
Let $0<x_0<1$, $\nu$ be a probability measure on $\Sigma$ satisfying $\nu(\Sigma_{1-x_0})=1$  and $\hat\nu$ be the   \mbox{1-measure} satisfying $d\hat\nu(x)=d\nu(\frac{x}{1-x_0})$.  If 
\begin{equation*} 
\lim_{n\to \infty} \mathbb{P}\left(\sigma^{\hat\nu}_n(1)=1\right)=0,
\end{equation*}
then $\forall s \in \mathbb{R}$,
\begin{equation*} 
\lim_{n\to \infty} \mathbb{P}\left(\frac{\ell(\sigma^\nu_n)-x_0n}{\sqrt{x_0(1-x_0)n}}\leq s\right)=\int_{-\infty}^s\frac{1}{\sqrt{2\pi}} e^{\frac{-x^2}{2}}dx.
\end{equation*}
\end{conjecture}
One bound is simple to prove by the remark above.

\paragraph*{}
 A possible generalization of the Ewens distributions is the following. 
\begin{definition}\label{gewens}
Let  $\hat{\theta}=(\hat{\theta}_{i})_{i\geq 1}$ be  a sequence of positive real numbers, we say that $\sigma_n$ follows  the generalized Ewens distribution on $\mathfrak{S}_n$  with parameter ${\hat{\theta}}$ if for all $\sigma \in \mathfrak{S}_n$,
\begin{align*}
{\mathbb{P}}(\sigma_n=\sigma)= \frac{\prod_{i\geq 1}\hat{\theta}_i^{r_i(\sigma)}}{\sum_{\sigma\in \mathfrak{S}_n}\prod_{i\geq 1}\hat{\theta}_i^{r_i(\sigma)}}.
\end{align*}
Here, $r_i(\sigma)$ is the number of cycles of $\sigma$ of length $i$.
\end{definition}
This generalization was studied in some cases in details by \cite*{2011arXiv1102.4796E}. In the general case, it is not obvious to have a good control on the number of cycles. Nevertheless, by using some results of Ercolani and Ueltschi, we can conclude in some cases.
\begin{corollary} \label{2.2} 
Let $(\sigma_n)_{n\geq 1}$ be a sequence of random permutations such that for all positive integer $n$, $\sigma_n$ follows the generalized Ewens distribution  with parameter $\hat{\theta}=(\hat{\theta}_{i})_{i\geq 1}$.
Assume that $\hat{\theta}$ satisfies one of the following hypotheses: 
\begin{itemize}
\item $\hat\theta_i=e^{i^\gamma},\gamma>1$,
\item$ \lim_{i\to \infty} \sum_{k=1}^{i-1}\frac{\hat\theta_{k}\hat\theta_{i-k}}{\hat\theta_i}=0$,
\item$ \lim_{i\to \infty} \hat\theta_i=\theta$,
\item $ \lim_{i\to \infty} \frac{\hat\theta_i}{i^\gamma}=1,$ where $ 0\leq \gamma<\frac{1}{7}$, 
\item $\hat\theta_i=i^\gamma$, $ \gamma<-1$.
\end{itemize}
Then we have Tracy-Widom fluctuations \eqref{TW} and the convergence at the edge to the Airy ensemble \eqref{TW2}.
\end{corollary}
For the descent process, we have the convergence for a larger class of parameters.
\begin{corollary} \label{4.2}Let $(\sigma_n)_{n\geq 1}$ be a sequence of random permutations such that for all positive integer $n$, $\sigma_n$ follows the generalized Ewens distribution  with parameter $\hat{\theta}=(\hat{\theta}_{i})_{i\geq 1}$.
Assume that $\hat{\theta}$ meets one of the hypotheses of the previous corollary or 
$ \lim_{i\to \infty} \frac{\hat\theta_i}{i^\gamma}=1 ,$ where $ \gamma\geq 0$. 
We have then the convergence of $D(\sigma_n)$ to the determinantal point process with kernel $K_0$ \eqref{main}.
\end{corollary}
Corollaries \ref{2.2} and \ref{4.2} are a direct application from the computations of Ercolani and Ueltschi. In particular, we use the following results: 
 \begin{lemma}
Let $\hat{\theta}=\{\hat{\theta_i}\}_{i\geq 1}$ and $\{\sigma_n\}_{n\geq1}$  be a sequence of random permutations following the generalized Ewens distribution with parameter $\hat\theta$.
\begin{itemize}
\item If $\hat\theta_i=e^{i^\gamma}$ with $\gamma>1$, then $\#(\sigma_n)\overset{\mathbb{P}}\to 1$ \citep*[Theorem 3.1]{2011arXiv1102.4796E}.
\item If $\hat{\theta_i}\to\theta$,  then $\frac{1}{\theta \log(n)}\mathbb{E}(\#(\sigma_n))\to 1$ \citep*[Theorem 6.1]{2011arXiv1102.4796E}.
\item If $\hat{\theta_i}=i^{-\gamma}$   with $\gamma>1$, then $\#(\sigma_n) \overset{d}\to 1+\sum_{i}Poisson\{\theta_i\}$ \citep*[Theorem 7.1]{2011arXiv1102.4796E}.
\item If $ \sum_{k=1}^{n-1} \frac{\hat\theta_k\hat\theta_{n-k}}{\hat\theta_n}\to 0$,   then $\#(\sigma_n)\overset{\mathbb{P}}\to 1$ \citep*[Theorem 3.1]{2011arXiv1102.4796E}.
\item If $\frac{\hat{\theta_i}}{i^{\gamma}}\to 1$ with $\gamma>0$, then
$\lim_{n\to\infty} n^\frac{-\gamma}{\gamma+1}\mathbb{E}(\#(\sigma_n))= \left(\frac{\Gamma(\gamma)}{\gamma^\gamma}\right)^\frac{1}{\gamma+1}
$ \citep*[Theorem 5.1]{2011arXiv1102.4796E}.
\end{itemize}
 \end{lemma}
 Using this lemma, it is obvious that \eqref{h2} is satisfied under the  assumptions of Corollary \ref{2.2}.
Moreover, \eqref{h3} can be replaced by 
\begin{align*}
\lim_{n\to \infty} \mathbb{E}\left(\frac{\#(\sigma_n)}{n}\right) \to 0.
\end{align*}
This result is a consequence of the stability under conjugation. Indeed,
\begin{align*}
\mathbb{P}(\sigma_n(1)=1)=\frac{1}{n} \sum_{i=1}^n\mathbb{P}(\sigma_n(i)=i) \leq \mathbb{E}\left(\frac{\#(\sigma_n)}{n}\right).
\end{align*}
Using this observation, it is obvious that \eqref{h3} is satisfied under assumptions of Corollary \ref{4.2}. 
\paragraph*{} \cite{pitman1992two} introduced a two-parameters  generalization of the Ewens distribution. Using the same notations as in \citep{pitman1992two},  we can apply Theorems \ref{Airyens} for $\alpha<\frac{1}{6}$ and Theorem \ref{VCthm} for $\alpha<1$.
\paragraph*{} The bound $n^\frac 16$ of Theorem \ref{the1}  may not be optimal. The best counterexample we found is when the number of cycles is of  order  $\sqrt{n}$ for the general case and of order  $n$ for virtual random  permutations.
Nevertheless, using the same lines of proof, we can obtain  the convergence of $\frac{\ell(\sigma_n)}{\sqrt{n}}$ with optimal hypotheses.
\begin{proposition} 
Assume that the sequence of random permutations  $(\sigma_n)_{n\geq 1}$ satisfies \eqref{h1} and  the number of cycles is such that: For all $\varepsilon>0$,
\begin{equation*}
\lim_{n\to \infty}\mathbb{P}\left(\frac{\#(\sigma_n)}{{\sqrt{n}} }>\varepsilon\right) =0,
\end{equation*}
then  $\forall\varepsilon>0$, \begin{equation*} 
\lim_{n\to\infty} \mathbb{P}\left(\left|\frac{\ell(\sigma_n)}{\sqrt{n}} - 2\right|>\varepsilon\right)=\lim_{n\to\infty} \mathbb{P}\left(\left|\frac{\underline{\ell}(\sigma_n)}{\sqrt{n}} - 2\right|>\varepsilon\right)=0.
\end{equation*}
\end{proposition}
In this case, the bound $\sqrt{n}$ in the second condition is optimal.

\section{Proof of results}
\subsection{Proof of Theorem \ref{the1}}
\label{proof1}
The key argument of our  proof is the following lemma:
\begin{lemma} \label{lem}
For any permutation $\sigma$ and for any transposition $\tau$, 
\begin{equation*}
|\ell(\sigma \circ \tau )-\ell(\sigma)|\leq 2,
\quad |\underline \ell(\sigma)-\underline \ell(\sigma \circ \tau )|\leq 2.
\end{equation*}

\end{lemma}
\begin{proof}
Let $\sigma$ be  a permutation.  By definition of $\ell(\sigma)$,  there exists ${i_1<i_2<\dots<i_ {\ell(\sigma)}}$ such that ${\sigma(i_1)<\dots<\sigma(i_{\ell(\sigma)})}$. Let $\tau=(j,k)$  be a transposition
and $i'_1,i'_2,\dots,i'_m$ be   the same sequence as $i_1,i_2,\dots, i_{\ell(\sigma)}$ after removing $j$ and $k$ if needed. We have  $\sigma(i'_1)<\dots<\sigma(i'_{m})$. In particular,  $\ell(\sigma)-2\leq m\leq \ell(\sigma)$. Knowing that  $\forall i\notin \{j,k\}$, $\sigma\circ\tau (i)=\sigma(i)$, then  $$\sigma\circ \tau (i'_1)<\dots<\sigma\circ \tau (i'_{m}).$$ Therefore,
\begin{align*}
\ell(\sigma)-\ell(\sigma \circ \tau )\leq 2.
\end{align*}
We obtain the second inequality by replacing $\sigma$ by $\sigma \circ \tau$.
 For $\underline{\ell}(\sigma)$ the proof is similar.
\end{proof}
 Let $\sigma_n$ be a random permutation stable under conjugation. To prove Theorem \ref{the1}, the idea is to modify $\sigma_n$ to obtain a random permutation  stable under conjugation with only one cycle. We define  the following Markov operator T.  If  the realisation  $\sigma$ of $\sigma_n$ has one cycle, $\sigma$ remains unchanged ($T(\sigma)=\sigma$). Otherwise, we  choose with uniform probability two different cycles $C_1$ and $C_2$, and then independently two elements $i \in C_1$ and $j \in C_2$ uniformly within each cycle. In this case, $T(\sigma)=\sigma\circ (i,j)$. 
For example,  for $n=3$,   transitions' probabilities of $T$ are  given in Figure \ref{figm}.  
\begin{figure}[H]
\centering
\begin{tikzpicture}

    \node[state] (s1)  {Id};
    \node[state, below=1cm of s1] (t1) {$(1,2)$};
    \node[state, right=3cm of t1] (t2) {$(2,3)$};
	\node[state, left=3cm of t1] (t3) {$(1,3)$};
	\node[state, below=1cm  of t3] (c1) {$(1,2,3)$};
	\node[state, right= 7cm of c1] (c2) {$(1,3,2)$};
        \draw[every loop,
        line width=0.3mm,
        auto=left,
        >=latex,
        ]
            (s1) edge[]  node {$\frac{1}{3}$} (t1)
             (s1) edge[]  node {$\frac{1}{3}$} (t3)
              (s1) edge[]  node {$\frac{1}{3}$} (t2)
                 (t1) edge[ ]  node {$\frac{1}{2}$} (c1)
       (t2) edge[ ]  node {$\frac{1}{2}$} (c1)
                 (t3) edge[ ]  node {$\frac{1}{2}$} (c1)
                 (t1) edge[]  node {$\frac{1}{2}$} (c2)
                 (t2) edge[]  node {$\frac{1}{2}$} (c2)
                 (t3) edge[]  node {$\frac{1}{2}$} (c2)
                 (c2) edge[loop below]  node {1} (c2)
                 (c1) edge[loop below]  node {1} (c1);
    \end{tikzpicture}
    \caption{The transition probabilities of $T$ on $\mathfrak{S}_3$}
    \label{figm}
\end{figure}
We denote by $T^k(\sigma_n)$ the random permutation obtained after  applying $k$ times the operator $T$. Table \ref{T1} sums up distributions after  different steps if we start from the uniform distribution on $\mathfrak{S}_3$. 
\begin{table}[H]
\centering
\begin{tabular}{|l|l|l|l|}
\hline
          & $\sigma_3$      & $T(\sigma_3)$      & $T^2(\sigma_3)$   \\ \hline
Id        & $1/6$ & $0$            & $0$           \\ \hline
$(1,2)$   & $1/6$ & $1/18$ & $0$           \\ \hline
$(1,3)$   & $1/6$ & $1/18$ & $0$           \\ \hline
$(2,3)$   & $1/6$ & $1/18$ & $0$           \\ \hline
$(1,2,3)$ & $1/6$ & $5/12$ & $1/2$ \\ \hline
$(1,3,2)$ & $1/6$ & $5/12$ & $1/2$ \\ \hline
\end{tabular}
\caption{Transitions for the uniform setting }
\label{T1}
\end{table}
Note that for all positive integer $i<n$,  
\begin{equation}\label{cs1}
\#(T^{i}(\sigma_n))\overset{a.s}{=}\max(\#(\sigma_n)-i,1).
\end{equation}
\begin{lemma}
\label{lemmma11}
If $(\sigma_n)_{n\geq 1}$ is stable under conjugation, then for all positive integer $n$, the law of $T^{n-1}(\sigma_n)$  is the uniform distribution on the set of permutations with a unique cycle. More formally,
\begin{align*}
\mathbb{P}\left(T^{n-1}(\sigma_n)=\sigma\right)=\frac{1}{(n-1)!}\mathbbm{1}_{\#({\sigma})=1}.
\end{align*}
\end{lemma}
\begin{proof}
 First, by construction, if $\sigma_n$ is stable under conjugation, $T(\sigma_n)$ is also stable under conjugation. Indeed, if $\hat{\sigma}_1$, $\hat{\sigma}_2 \in \mathfrak{S}_n$ then
\begin{align*}
\mathbb{P}(T(\sigma_n)=\hat{\sigma}_1)&=\sum_{\sigma\in \mathfrak {S}_n} \sum_{i<j} \left( \frac{\mathbbm{1}_{\#(\hat{\sigma}_1)=\#(\sigma)-1}\mathbbm{1}_{\sigma^{-1}\circ\ \hat{\sigma}_1=(i,j)}}{\mathcal{C}_\sigma(i)\mathcal{C}_\sigma(j) {{\#(\sigma)}\choose{2}}} +\mathbbm{1}_{\#(\sigma)=1}\mathbbm{1}_{\sigma=\hat{\sigma}_1}\right)\mathbb{P}(\sigma_n=\sigma)
\\&=\sum_{\sigma\in \mathfrak {S}_n} \sum_{i<j} \left( \frac{\mathbbm{1}_{\#(\hat\sigma_2\circ\hat{\sigma}_1\circ\hat\sigma_2^{-1})=\#(\sigma)-1}\mathbbm{1}_{\hat\sigma_2\circ\sigma^{-1}\circ\ \hat{\sigma}_1\circ\hat\sigma_2^{-1}=(\hat\sigma_2(i),\hat\sigma_2(j))}}{\mathcal{C}_{\hat\sigma_2\circ \sigma\circ\hat\sigma_2^{-1}}(\hat\sigma_2(i))\mathcal{C}_{\hat\sigma_2\circ \sigma\circ\hat\sigma_2^{-1}}(\hat\sigma_2(j)){{\#(\hat\sigma_2\circ \sigma\circ\hat\sigma_2}\choose{2}}} +\mathbbm{1}_{\#(\hat\sigma_2\circ{\sigma}\circ\hat\sigma_2^{-1})=1}\mathbbm{1}_{\hat\sigma_2\circ{\sigma}\circ\hat\sigma_2^{-1}=\hat\sigma_2\circ\hat{\sigma}_1\circ\hat\sigma_2^{-1}}\right)\\&\times\mathbb{P}(\sigma_n=\hat\sigma_2\circ{\sigma}\circ\hat\sigma_2^{-1})
\\
&=\sum_{\sigma\in \mathfrak {S}_n} \sum_{i<j}\left( \frac{\mathbbm{1}_{\#(\hat\sigma_2\circ\ \hat{\sigma}_1\circ\hat\sigma_2^{-1})=\#(\sigma)-1}\mathbbm{1}_{\sigma^{-1}\circ\ \hat\sigma_2\circ\ \hat{\sigma}_1\circ\hat\sigma_2^{-1}=(i,j)}}{\mathcal{C}_\sigma(i)\mathcal{C}_\sigma(j){{\#(\sigma)}\choose{2}}} +\mathbbm{1}_{\#(\sigma)=1}\mathbbm{1}_{\sigma=\hat\sigma_2\circ\ \hat{\sigma}_1\circ\hat\sigma_2^{-1}}\right)\mathbb{P}(\sigma_n=\sigma)
\\&=
\mathbb{P}(T(\sigma_n)=\hat\sigma_2\circ\hat{\sigma}_1\circ\hat\sigma_2^{-1})
,
\end{align*}
where $\mathcal{C}_\sigma(i)$ is the length  of the cycle of $\sigma$ containing $i$. In particular, the law of  $T^{n-1}(\sigma_n)$ is stable under conjugation.
Moreover, using  \eqref{cs1},   
\begin{align} \label{cs2}
\#(T^{n-1} \left(\sigma_{n})\right)\overset{a.s}{=}\max(\#(\sigma_n)-n+1,1)=1. \end{align}  
Knowing that all elements of $\mathfrak{S}_n$ with a unique cycle belong to the same class of conjugation, they are equally distributed and  Lemme~\ref{lemmma11} follows from  \eqref{cs2}.
\end{proof}
The previous Lemma is equivalent to say that $T^{n-1}(\sigma_n)$ follows the  Ewens distribution on $\mathfrak{S}_n$ with parameter 
$\theta=0$.
\begin{proof}[Proof of Theorem \ref{the1}]
Equality \eqref{cs1} implies that 
$T^{n-1}(\sigma_n)\overset{a.s}= T^{\#(\sigma_n)-1}(\sigma_n)$. Therefore using Lemma \ref{lem}, we obtain  almost surely that:
\begin{align*}
|\ell(T^{n-1}(\sigma_n))-\ell(\sigma_n)|=
|\ell(T^{\#(\sigma_n)-1}(\sigma_n))-\ell(\sigma_n)|
\leq  2(\#(\sigma_n)-1).
\end{align*}
Thus, if $\sigma_n$ satisfies the hypothesis  \eqref{h2}, then $\forall \varepsilon >0 $, 
\begin{align}\label{end1}  \mathbb{P}\left(\left|\frac{\ell(T^{n-1}(\sigma_n))-\ell(\sigma_n)}{n^\frac{1}{6}} \right|> \varepsilon\right) =0.\end{align}  
\paragraph*{} Using Lemma \ref{lemmma11}, $T^{n-1}(\sigma_n)$ does not depend on the law of $\sigma_n$. Therefore,  it is enough to prove Theorem \ref{the1} for one particular case. In fact, the convergence \eqref{TW} has been obtained for the uniform setting, see Theorem \ref{dbj}. By choosing $(\sigma_n)_{n\geq 1}$ a sequence of random permutations following the uniform distribution,  we have then  \eqref{TW}  for the Ewens distribution with parameter $\theta=0$. For the general case, if the sequence $(\sigma_n)_{n\geq 1}$ satisfies \eqref{h1} and \eqref{h2}, we can conclude using Lemma \ref{lemmma11} and \eqref{end1}.

The same argument can be applied for the length of  longest decreasing subsequence.  
 \end{proof}
\subsection{Proof of results related to the Robinson–Schensted transform of random permutations}
\label{RSKPROOF}
\paragraph*{} To prove Theorems \ref{Airyens} and \ref{VCthm} we need to recall a well-known property of the  Robinson–Schensted correspondence. Let $\sigma \in \mathfrak{S}_n$. 
We denote  \begin{align*}
\mathfrak{I}_1(\sigma):&=\{s\subset\{1,2,\dots,n\};\; \forall i,j \in s,\; (i-j)(\sigma(i)-\sigma(j))\geq 0 \},
\\ \mathfrak{D}_1(\sigma):&=\{s\subset\{1,2,\dots,n\};\; \forall i,j \in s,\; (i-j)(\sigma(i)-\sigma(j))\leq 0 \},
\\\mathfrak{I}_{k+1}(\sigma):&=\{s\cup s',\; s\in \mathfrak{I}_k,\;s'\in \mathfrak{I}_1\},
\\ \mathfrak{D}_{k+1}(\sigma):&=\{s\cup s',\; s\in \mathfrak{D}_k,\;s'\in \mathfrak{D}_1\}.
\end{align*}
We have then
\begin{lemma} \label{RSKLEMMA} \citep*{GREENE1974254}
For any permutation $ \sigma\in \mathfrak{S}_n$,
\begin{align*}
\max_{s\in \mathfrak{I}_i(\sigma)} |s| =\sum_{k=1}^i \lambda_k(\sigma), \quad
\max_{s\in \mathfrak{D}_i(\sigma)} |s| =\sum_{k=1}^i \lambda'_k(\sigma).
\end{align*}
\end{lemma}
In particular, $$\max_{s\in \mathfrak{I}_1(\sigma)} |s| =\lambda_1(\sigma)=\ell(\sigma),\quad \max_{s\in \mathfrak{D}_1(\sigma)} |s| =\lambda'_1(\sigma)=\underline{\ell}(\sigma).$$
\paragraph*{} This result is proved first by \cite{GREENE1974254} (see also \citep*[Theorem 3.7.3]{Sagan2001}). It will be the keystone to prove Theorem \ref{Airyens} and Theorem \ref{VCthm} as it implies  the following lemma   which is the counterpart of Lemma~\ref{lem}.
\begin{lemma} \label{lemma2}
For any permutation $\sigma$ and transposition  $\tau$,\begin{equation} \label{sum}
\left|\sum_{k=1}^i \lambda_k(\sigma)-{\lambda}_k\left(\sigma\circ\tau\right)\right| \leq 2, \quad
\left|\sum_{k=1}^i \lambda'_k(\sigma)-\lambda'_k\left(\sigma\circ\tau\right)\right| \leq 2.
\end{equation}
Moreover, 
\begin{equation} \label{sep}
\left|\lambda_i(\sigma)-\lambda_i\left(\sigma\circ\tau\right)\right| \leq 4, \quad
\left|\lambda'_i(\sigma)-\lambda'_i\left(\sigma\circ\tau\right)\right| \leq 4.
\end{equation}
\end{lemma}
\begin{proof} Let $\sigma$ be a permutation and $\tau=(l,m)$ be a transposition.  We have then for all integer $i$,
\begin{equation*}
\{s\setminus{\{l,m\}},s\in \mathfrak{I}_i(\sigma)\}\subset \mathfrak{I}_i(\sigma\circ\tau)
\end{equation*}
and similarly  
\begin{equation*}
\{s\setminus{\{l,m\}},s\in \mathfrak{D}_i(\sigma)\}\subset \mathfrak{D}_i(\sigma\circ\tau).
\end{equation*}
Consequently, by Lemma \ref{RSKLEMMA},
\begin{equation*}
\sum_{k=1}^i \lambda_k(\sigma)-{\lambda}_k(\sigma\circ\tau) \geq -2
, \quad \sum_{k=1}^i \lambda'_k(\sigma)-\lambda'_k(\sigma\circ\tau) \geq -2.
\end{equation*}
Using the same argument with $\sigma \circ \tau $ instead of $\sigma$, \eqref{sum}  follows. Moreover, since $$\lambda_{i+1}=\sum_{k=1}^{i+1}\lambda_k-\sum_{k=1}^i\lambda_k, \quad \lambda'_{i+1}=\sum_{k=1}^{i+1}\lambda'_k-\sum_{k=1}^i\lambda'_k,$$ 
the triangle inequality yields   \eqref{sep}.
\end{proof}
\begin{proof}[Proof of Theorem \ref{Airyens}]
Similarly to the proof of Theorem \ref{the1}, we will use the same Markov operator $T$ to compare our random permutation with the uniform distribution. Using Lemma \ref{lemma2} and the equality \eqref{cs1} we obtain  
\begin{equation} \label{sep2}
\left|\lambda_i(\sigma_n)-\lambda_i\left(T^{n-1}(\sigma_n)\right)\right| \leq 4(\#(\sigma_n)-1).
\end{equation}
Consequently, under \eqref{h2}, $\forall \varepsilon>0$,
\begin{equation}\label{fin2}
\lim_{n\to \infty} \mathbb{P}\left(\left|\frac{\lambda_i(\sigma_n)-\lambda_i\left(T^{n-1}(\sigma_n)\right)}{n^\frac{1}{6}}\right| >\varepsilon \right)= 0.
\end{equation}
The remainder of the proof is similar to the proof of Theorem \ref{the1}.
\end{proof} 
\paragraph*{}We will now prove Theorem \ref{VCthm}.
\paragraph*{}
Let $(O,\vec{x},\vec{y})$ be  the canonical frame of the Euclidean plane and $\vec{u}:=\frac{\sqrt{2}}{2}(\vec{x}+\vec{y})$, $\vec{v}:=\frac{\sqrt{2}}{2}(\vec{y}-\vec{x})$. Let $\lambda \in \mathbb{Y}_n$. Using the convention $\lambda_0=\infty$, let $\mathscr{C}_\lambda$ be the curve obtained by connecting  the points with coordinates   $(0,\lambda_0),(0,\lambda_1), (1,\lambda_1),(1,\lambda_2),\dots,$ $ (i,\lambda_{i}),(i,\lambda_{i+1}),\dots$ in the axes system $(O,\overrightarrow{u},\overrightarrow{v})$ as in Figure \ref{figL31}. By construction $\mathscr{C}_\lambda$ is the curve of $L_\lambda$. This yields the following.
\begin{lemma} \label{lemmainq}
Let  $\alpha,\beta \in \mathbb{N}$ and $A$ the point  such that  $\overrightarrow{OA}=\alpha\vec{u}+\beta\vec{v}$. If $A\in \mathscr{C}_\lambda$, then
\begin{equation}\label{la1}
\lambda_{\alpha+1}\leq \beta\leq \lambda_\alpha.
\end{equation}
\end{lemma}
\begin{figure}[H]
\centering
\begin{tikzpicture}[/pgfplots/y=0.45cm, /pgfplots/x=0.45cm]
      \begin{axis}[
    axis x line=center,
    axis y line=center,
    xmin=0, xmax=7,
    ymin=0, ymax=9, clip=false,
    ytick={0},
	xtick={0},
    minor xtick={0,1,2,3,3,4,5,6,7,8,9},
    minor ytick={0,1,2,3,3,4,5,6,7,8,9},
    grid=both,
    legend pos=north west,
    ymajorgrids=false,
    xmajorgrids=false, anchor=origin,
    grid style=dashed    , rotate around={-90:(rel axis cs:0,0)},
]

\addplot[
    color=blue,
        line width=3pt,
    ]
    coordinates {
    (0,8)(0,7)(1,7)(1,5)(2,5)(2,2)(3,2)(3,1)(5,1)(5,0)(6,0)
    };
\end{axis}
    \end{tikzpicture}\;  \;  \; \; \;  \; \; \;  
    \begin{tikzpicture}[/pgfplots/y=0.45cm, /pgfplots/x=0.45cm]

 \begin{axis}[
    axis x line=center,
    axis y line=center,
    xmin=0, xmax=7,
    ymin=0, ymax=9, clip=false,
    ytick={0},
	xtick={0},
    minor xtick={0,1,2,3,3,4,5,6,7,8,9},
    minor ytick={0,1,2,3,3,4,5,6,7,8,9},
    grid=both,
    legend pos=north west,
    ymajorgrids=false,
    xmajorgrids=false, 
    grid style=dashed    , rotate around={45:(rel axis cs:0,0)},
]

\addplot[
    color=blue,
        line width=3pt,
    ]
    coordinates {
    (0,8)(0,7)(1,7)(1,5)(2,5)(2,2)(3,2)(3,1)(5,1)(5,0)(6,0)
    };
   \addplot[cyan,
        quiver={u=\thisrow{u},v=\thisrow{v}},
        -stealth]
    table
    {
    x y u v
    0 0 1 0
    0 0 0 1
    };
        \node[cyan] at (axis cs: -0.4,1) {$\vec{v}$};
 \node[cyan] at (axis cs: 1,-0.4) {$\vec{u}$};
                 \node[cyan] at (axis cs: -0.2,-0.2) {$o$};
\end{axis}
    \end{tikzpicture}
    \caption{  $\mathscr{C}_\lambda$ for $\lambda=(7,5,2,1,1,\underline{0})$}
     \label{figL31}
\end{figure}
We have also the following result.
\begin{lemma} \label{lemmapari}
For all $i\in\mathbb{Z}$,
\begin{align}\label{eq15}
\frac{\sqrt{2}}{2} L_\lambda\left(\frac{\sqrt{2}}{2}i\right)\pm\frac{i}{2}\in \mathbb{N},
\end{align}
\end{lemma}
\begin{proof}
Let $M$ be such that $ \overrightarrow{OM}=s_1 \vec{u}+s_2\vec{v}$. By construction, if $M\in\mathscr{C}_\lambda$ then  $s_1,s_2\geq 0$ and either $s_1 \in \mathbb{N}$ or $s_2 \in \mathbb{N}$. If we apply this observation to $M$ defined by  $$\overrightarrow{OM}:=\frac{\sqrt{2}}{2}i\vec{x}+L_\lambda\left(\frac{\sqrt{2}}{2}i\right)\vec{y}=\left(\frac{\sqrt{2}}{2} L_\lambda\left(\frac{\sqrt{2}}{2}i\right)+\frac{i}{2}\right)\overrightarrow{u}+\left(\frac{\sqrt{2}}{2} L_\lambda\left(\frac{\sqrt{2}}{2}i\right)-\frac{i}{2}\right)\overrightarrow{v},$$ we obtain \eqref{eq15}.
\end{proof}
To prove Theorem \ref{VCthm}, our main lemma is the following.
\begin{lemma} \label{34}
Let $n,m\in \mathbb{N}^*$, $\lambda=(\lambda_i)_{i\geq1}\in\mathbb{Y}_n$, $\mu=(\mu_i)_{i\geq1}\in\mathbb{Y}_m$. Then,
\begin{align}\label{disVC}
\sup_{s\in\mathbb{R}}\left(L_\lambda(s)-L_\mu(s)\right)^2\leq 4 \max_{i\geq 1} \left|\sum_{k=1}^i( \lambda_k-\mu_k)\right|.
\end{align}
\end{lemma}

\begin{figure}[H]
\centering
\begin{tikzpicture}    [/pgfplots/y=0.8cm, /pgfplots/x=0.8cm]
      \begin{axis}[
    axis x line=center,
    axis y line=center,
    xmin=0, xmax=8,
    ymin=0, ymax=8, clip=false,
    ytick={0},
	xtick={0},
    minor xtick={0,1,2,3,3,4,5,6,7,8,9},
    minor ytick={0,1,2,3,3,4,5,6,7,8,9},
    grid=both,
    legend pos=north west,
    ymajorgrids=false,
    xmajorgrids=false, anchor=origin,
    grid style=dashed    , rotate around={45:(rel axis cs:0,0)},
]
\addplot[
    color=blue,
        line width=3pt,
    ]
    coordinates {
    (0,8)(0,7)(1,7)(1,5)(2,5)(2,2)(3,2)(3,1)(5,1)(5,0)(8,0)
    };
    \addplot[
    color=green,
        line width=2pt,
    ]
    coordinates {
    (0,8)(0,4)(1,4)(1,4)(2,4)(2,3)(5,3)(5,1)(6,1)(6,0)(8,0)
    };
    \addplot[cyan,
        quiver={u=\thisrow{u},v=\thisrow{v}},
        -stealth]
    table
    {
    x y u v
    0 0 1 0
    0 0 0 1
    };
        \node[cyan] at (axis cs: 0.2,0.8) {$\vec{v}$};
 \node[cyan] at (axis cs: 0.7,0.2) {$\vec{u}$};
                 \node[] at (axis cs: -0.2,-0.2) {$o$};

       \addplot [ mark=*, color=magenta] table {
0 4
1 5
};
    \node[color=magenta] at (axis cs: 0.7,5) {$k_{-4}=-1$};

       \addplot [ mark=*, color=magenta] table {
3 1
5 3
};
    \node[color=magenta] at (axis cs:4,1.8) {$k_2=2$};         
\end{axis}
\begin{axis}[
	axis x line=center,
    axis y line=center,
    xmin=-6, xmax=6,
    ymin=0, ymax=8, anchor=origin, clip=false,
    xtick={-7,-6,-5,-4,-3,-2,-1,0,1,2,3,4,5,6,7},
    ytick={0,1,2,3,3,4,5,6,7,8},
    legend pos=north west,
    ymajorgrids=false,
    xmajorgrids=false,rotate around={0:(rel axis cs:0,0)},
    grid style=dashed,
    ];
       
       \addplot[red,
        quiver={u=\thisrow{u},v=\thisrow{v}},
        -stealth]
    table
    {
    x y u v
    0 0 1 0
    0 0 0 1
    };
        \node[red] at (axis cs: 0.2,1) {$\vec{y}$};
 \node[red] at (axis cs: 0.8,0.25) {$\vec{x}$};

\end{axis}
    \end{tikzpicture}
    \caption{  An example  where $\lambda= {(7,5,2,1,1,\underline{0})}$ and $\mu={{(4,4,3,3,3,1,\underline{0})}}$}
     \label{figL2}
\end{figure}
\begin{proof}
Note that for any $i\in \mathbb{Z}$, $s\mapsto L_{\lambda}(s)$  and  $s\mapsto L_{\mu}(s)$ are affine functions on  $\left[\frac{\sqrt{2}}{2}i,\frac{\sqrt{2}}{2}(i+1)\right]$ 
and thus \eqref{disVC} is equivalent to 
\begin{equation*}
\sup_{i\in\mathbb{Z}}\left(L_\lambda\left(\frac{\sqrt{2}}{2}i\right)-L_\mu\left(\frac{\sqrt{2}}{2}i\right)\right)^2\leq 4  \max_{i\geq 1} \left|\sum_{k=1}^i( \lambda_k-\mu_k)\right|.
\end{equation*}
Let $i \in \mathbb{Z}$. It follows from Lemma \ref{lemmapari} that there exists $k_i \in \mathbb{Z}$ such that,  
\begin{align*}
L_{\mu}\left(\frac{\sqrt{2}}{2}i\right)- L_{\lambda}\left(\frac{\sqrt{2}}{2}i\right) = k_i\sqrt{2}.
\end{align*}
To simplify notations, we denote $$j:={\sqrt{2}L_{\lambda}\left(\frac{\sqrt{2}}{2}i\right)}.$$
Let $A$ and $B$ be the points such that 
\begin{align}
\overrightarrow{OA}:=\frac{\sqrt{2}}{2}(i\vec{x}+j\vec{y})=\frac{i+j}{2}\vec{u}+\frac{j-i}{2}\vec{v}, \qquad \overrightarrow{OB}=\frac{\sqrt{2}}{2}(i\vec{x}+(j+2k_i)\vec{y})=\frac{i+j+2k_i}{2}\vec{u}+\frac{j-i+2k_i}{2}\vec{v}.
\end{align}
 Clearly $A\in \mathscr{C}_\lambda$ and $B\in \mathscr{C}_\mu$.    By Lemma \ref{lemmapari}, $\frac{i+j}{2},\frac{j-i}{2} \in \mathbb{N}$. We can then apply Lemma \ref{lemmainq}.  In the case where $k_i>0$, we have
\begin{align*}
\lambda_{\frac{i+j}{2}+1} \leq \frac{j-i}{2}, \quad
\mu_{\frac{i+j}{2}+k_i} \geq \frac{j-i}{2}+k_i.
\end{align*}
Using the fact that $(\lambda_{l})_{l\geq 1}$ and  of $(\mu_{l})_{l\geq 1}$ are decreasing, we have, 
\begin{align*}
2 \max_{l\geq 1} \left|\sum_{k=1}^l( \lambda_k-\mu_k)\right| \geq \sum_{l={\frac{i+j}{2}+1}}^{\frac{i+j}{2}+k_i}\mu_l-\lambda_l\geq \sum_{l={\frac{i+j}{2}+1}}^{\frac{i+j}{2}+k_i}\mu_{\frac{i+j}{2}+k_i}-\lambda_{\frac{i+j}{2}+1}\geq k_i^2. 
\end{align*}
Similarly, in the case where $k_i<0$,
\begin{align*}
-2 \max_{l\geq 1} \left|\sum_{k=1}^l( \lambda_k-\mu_k)\right| \leq \sum_{l={\frac{i+j}{2}+1+k_i}}^{\frac{i+j}{2}}\mu_l-\lambda_l\leq \sum_{l={\frac{i+j}{2}+1+k_i}}^{\frac{i+j}{2}}\mu_{\frac{i+j}{2}+k_i+1}-\lambda_{\frac{i+j}{2}}\leq- k_i^2.  
\end{align*}
This yields 
\begin{align*}
4 \max_{i\geq 1} \left|\sum_{k=1}^i( \lambda_k-\mu_k)\right| \geq \max_{i\in\mathbb{Z}}  \left(\sqrt{2}k_i\right)^2=\sup_{s\in\mathbb{R}}\left(L_\lambda(s)-L_\mu(s)\right)^2.
\end{align*}
\end{proof}
\begin{proof}[Proof of Theorem \ref{VCthm}]
Using \eqref{cs1} and   Lemma \ref{lemma2},  we have almost surely,
\begin{align}
\max_{i\geq 1} \left|\sum_{k=1}^i\left( \lambda_k(\sigma_n)-\lambda_k\left(T^{n-1}(\sigma_n)\right)\right)\right|\leq 2(\#(\sigma_n)-1).
\end{align}
By Lemma \ref{34} we obtain 
\begin{equation}
\sup_{s\in \mathbb{R}} \frac{1}{\sqrt{2n}}\left|L_{\lambda(\sigma_n)}\left({s}{\sqrt{2n}}\right)- L_{\lambda(T^{n-1}(\sigma_n))}\left({s}{\sqrt{2n}}\right)\right| \leq 2 \sqrt{\frac{\#(\sigma_n)-1}{n}}.
\end{equation}
Under \eqref{H4}, $\forall \varepsilon>0$, 
\begin{align}\label{eq200}
\lim_{n\to \infty} \mathbb{P}
\left(\sup_{s\in \mathbb{R}} \frac{1}{\sqrt{2n}}\left|L_{\lambda(\sigma_n)}\left({\sqrt{2n}}\right)- L_{\lambda(T^{n-1}(\sigma_n))}\left({s}{\sqrt{2n}}\right)\right|<\varepsilon\right)=1.\end{align}
If $\sigma_n$ follows the uniform distribution, \eqref{VC} is obtained by \cite*{Vershik1985}, see Theorem  \ref{BOOJ2},  and consequently we have \eqref{VC} for the Ewens distribution with parameter $\theta=0$. For a random permutation $\sigma_n$ stable under conjugation  \eqref{h1}, $T^{n-1}(\sigma_n)$  follows the Ewens distribution with parameter $\theta=0$ and if $\sigma_n$ satisfies moreover \eqref{H4}, we can conclude using \eqref{eq200}.  
\end{proof}
\subsection{Proofs of the applications  to virtual permutations}
\label{appli} 
\paragraph*{}
We will prove in this subsection Corollaries \ref{2.4}, \ref{2.4p}, {\ref{2.5p}} and {\ref{2.5}} and Proposition \ref{them11}.
We will not give details of the  proof of   Corollary {\ref{2.6}} because it is a   direct application of Proposition \ref{them11}.
\paragraph*{}
 We can have a combinatorial interpretation of \eqref{defdef}. Let $x=(x_i)_{i\geq1} \in \Sigma$. At the beginning, we have  an infinite number of circles  $\{C_n\}_{n\in\mathbb{Z}}$.
At each step $n\geq 1$ we choose an integer $pos_n$  with probability distribution  $\sum_{j\geq 1}x_j\delta_j+(1-\sum_{i\geq 1}x_i)\delta_0$ independently from the past. We insert then the number $n$  uniformly on the circle $C_{pos_n}$ if $pos_n>0$ and on the circle $C_{-n}$ if $pos_n=0$ . At each step, one reads the elements on each non-empty circle  counterclockwise to get a cycle. For example, if $pos_1=4$, $pos_2=1$, $pos_3=4$, $pos_4=0$ and $pos_5=0$, we obtain the permutation $(1,3)(2)(4)(5)$.
With this description, we have  
\begin{align*}
\mathbb{E}\left(\#\left(\sigma^{\delta_{x}}_n\right)\right)=n\left(1-\sum_{i\geq1}{x_i}\right)+\sum_{i=1}^\infty (1-(1-x_i)^n).
\end{align*}
\begin{proof}[Proof of Corollary \ref{2.4} and Corollary \ref{2.4p}]
In both corollaries, since $\sum_{i\geq1}x_i=1$, we have
\begin{align*}
\mathbb{E}\left(\#\left(\sigma^{\delta_{x}}_n\right)\right)=\sum_{i=1}^\infty (1-(1-x_i)^n).
\end{align*}
If $\alpha>6$, there exists a real number  $\beta$ such that   $\frac{5}{6(\alpha-1)}<\beta<\frac{1}{6}$. Moreover there exists $n_0$ such that $\forall n>n_0$, $x_n < n^{-\alpha}$. For any  $n >(n_0)^\frac{1}{\beta}$ and under hypothesis of Corollary \ref{2.4}, we have 
\begin{align*}
\mathbb{E}\left(\#\left(\sigma^{\delta_x}_n\right)\right)=\sum_{i=1}^\infty \left(1-(1-x_i)^n\right) &\leq n^{\beta}+n\sum_{[n^{\beta}]+1}^\infty n^{-\alpha}\\&\leq n^\beta + \frac{1}{\alpha-1} n \left(n^{\beta}\right)^{(-\alpha+1)}=o\left(n^\frac{1}{6}\right).
\end{align*}
Then Corollary \ref{2.4} follows from Theorem \ref{Airyens}.
If $\alpha>1$ and under hypothesis of Corollary \ref{2.4p}, there exists $n_0$ such that $\forall n>n_0$, $x_n < n^{-\alpha}$ and let $n >(n_0)^{\frac{1}{\alpha}}$ we have 
\begin{align*}
\mathbb{E}\left(\#\left(\sigma^{\delta_x}_n\right)\right)=\sum_{i=1}^\infty \left(1-(1-x_i)^n\right) &\leq n^\frac{1}{\alpha}+n\sum_{[n^{\frac 1 \alpha }]+1}^\infty n^{-\alpha}\\&\leq n^\frac{1}{\alpha} + \frac{1}{\alpha-1} n \left(n^{\frac 1 \alpha }\right)^{(-\alpha+1)}=o(n).
\end{align*}
Then Corollary \ref{2.4p} follows from Theorem \ref{VCthm}.
\end{proof}
\begin{proof}[Proof of Corollary \ref{2.5p} and Corollary \ref{2.5}]
\begin{align*}
\mathbb{E}(\#(\sigma_n^\nu))&=\sum_{\sigma\in \mathfrak{S}_n}\left(\#(\sigma)\int_{x\in \Sigma_1} f(n,x,\sigma) d\nu(x)\right)
\\&=\int_{x\in \Sigma_1}\sum_{\sigma\in \mathfrak{S}_n}\#(\sigma) f(n,x,\sigma) d\nu(x)
\\&=\int_{x\in \Sigma_1}\sum_{i=1}^\infty \left(1-(1-x_i)^n\right)d\nu(x).
\end{align*}
Therefore, we obtain  Corollary \ref{2.5p}  thanks to Theorem \ref{Airyens} .
\paragraph{}
Moreover, $\int_{x\in \Sigma_1}\sum_{i=1}^\infty \left(1-(1-x_i)^n\right)d\nu(x)=o(n)$ is always satisfied. Indeed,  we have for any   $0\leq y\leq 1$ and  $n\geq1$,
$$1-(1-y)^n \leq ny.$$ 
Let $x=(x_i)_{i\geq1}\in\Sigma$. Fix $\varepsilon>0$.  Since $\sum_{i=1}^\infty x_i \leq 1 $, there exists $n_0$ such that $\sum_{i=n_0+1}^\infty x_i <\varepsilon$. Then 
$$ \frac{1}{n}\sum_{i=1}^\infty (1-(1-x_i)^n)\leq \frac{1}{n}\sum_{i=1}^{n_0} 1+\sum_{i=n_0+1}^\infty x_i\leq  \frac{n_0}{n}+\varepsilon.$$
So that for any  $x=(x_i)_{i\geq1}\in\Sigma$, 
$$\lim_{n\to \infty}  \frac{1}{n}\sum_{i=1}^\infty (1-(1-x_i)^n)=0.$$ 
Since $\frac{1}{n}\sum_{i=1}^\infty (1-(1-x_i)^n)\leq 1$, we can conclude  using the dominated convergence theorem that $$\lim_{n\to \infty}\int_{x\in \Sigma_1}  \frac{1}{n}\sum_{i=1}^\infty (1-(1-x_i)^n)d\nu(x)=0.$$
Therefore, we obtain  \ref{2.5} thanks to Theorem \ref{VCthm}.
\end{proof}

\begin{proof} [Proof of Proposition \ref{them11}]
  An interpretation of the random permutation defined by equation \eqref{def2.1} is the following. Let $n$ be a positive integer. We construct  a subset $A_n$ of $\{1,2,\dots,n\}$ as follows: for every $1\leq i \leq n$, with probability $x_0$, $i \in A_n$  independently from other points. The points of $A_n$ are then fixed points of $\sigma_n$. After that,  we permute the elements of    $\{1,2,\dots,n\}\setminus A_n$ according to the probability distribution $\mathbb{P}_{n-|A_n|}$.\paragraph*{} The main idea is that  a decreasing subsequence cannot have more than one element belonging to $A_n$. Moreover, a decreasing subsequence of the restriction of $\sigma_n$ on  $\{1,2,\dots,n\}\setminus A_n$ is a decreasing subsequence of $\sigma_n$.
In other words, for all real  number $s$, for all $1\leq j\leq n$,
\begin{align*} 
\mathbb{P}_{j}(\{\sigma \in \mathfrak{S}_{j},\underline{\ell}(\sigma)\leq s-1\}) \leq \mathbb{P}( \underline{\ell}(\sigma_n)\leq s| |A_n|=n-j) \leq \mathbb{P}_{j}(\{\sigma \in \mathfrak{S}_{j},\underline{\ell}(\sigma)\leq s\}).
\end{align*}
More generally, using Lemma \ref{RSKLEMMA}, we have for all real numbers $s_1,\dots,s_{k}$, 
\begin{align*} 
\mathbb{P}_{j}(\{\sigma \in \mathfrak{S}_{j}, \forall i<k, \lambda'_i(\sigma)\leq s_i-2i+1\}) \leq \mathbb{P}(\forall i<k, \lambda'_i(\sigma_n)\leq s_i| |A_n|=n-j) \leq \mathbb{P}_{j}(\{\sigma \in \mathfrak{S}_{j},\forall i<k, \lambda'_i(\sigma)\leq s_i\}).
\end{align*}
Consequently,
\begin{align*} 
\mathbb{P}_{j}(\{\sigma \in \mathfrak{S}_{j}, \forall i<k, \lambda'_i(\sigma)\leq s_i-2k+1\}) \leq \mathbb{P}(\forall i<k, \lambda'_i(\sigma_n)\leq s_i | |A_n|=n-j) \leq \mathbb{P}_{j}(\{\sigma \in \mathfrak{S}_{j},\forall i<k, \lambda'_i(\sigma)\leq s_i\}).
\end{align*}
\paragraph*{}
 In the sequel of the proof, let $s_1,\dots,s_{k}$ be $k$ real numbers  and $\varepsilon >0$. As $|A_n|$ is a random binomial variable with parameters $n$ and $x_0$,  and using the central limit theorem, there exist $n_0$,  $\alpha>0$ such that, $n_0>\frac{\alpha^2}{(1-x_0)^2}$ and 
$\forall n>n_0$,
\begin{align}\label{eq7}
\mathbb{P}( ||A_n|-nx_0|< \alpha\sqrt{n}) >1-\varepsilon.
\end{align}
We denote by $p^n_j:=\mathbb{P}(|A_n|=n-j)$, $\tilde{x}_0:=1-x_0$, $\tilde{k}=2k-1$.
As
\begin{equation*} 
\mathbb{P}\left(\forall i \leq k,\frac{\lambda'_i(\sigma_n)-2\sqrt{n\tilde{x}_0}}{(n\tilde{x}_0)^\frac 16}\leq s_i\right)=\sum_{j=0}^n 
\mathbb{P}\left(\forall i \leq k,\frac{\lambda'_i(\sigma_n)-2\sqrt{n\tilde{x}_0}}{(n\tilde{x}_0)^\frac 16}\leq s_i\middle| |A_n|=n-j\right)p^n_j,
\end{equation*}
we have
\begin{align}\label{eq8}
\mathbb{P}\left(\forall i \leq k,\frac{\lambda'_i(\sigma_n)-2\sqrt{n\tilde{x}_0}}{(n\tilde{x}_0)^\frac 16}\leq s_i\right) \leq \varepsilon+\sum_{j=\ceil*{n\tilde{x}_0-\alpha\sqrt n} }^{\floor*{n\tilde{x}_0+\alpha\sqrt n} }\mathbb{P}_j\left(\left\{\sigma\in \mathfrak{S}_j,\forall i \leq k,\frac{\lambda'_i(\sigma)-2\sqrt{n\tilde{x}_0}}{(n\tilde{x}_0)^\frac 16}\leq s_i\right\}\right)p^n_j
\end{align}
and 
\begin{align}\label{eq9}
\mathbb{P}\left(\forall i \leq k,\frac{\lambda'_i(\sigma_n)-2\sqrt{n\tilde{x}_0}}{(n\tilde{x}_0)^\frac 16}\leq s_i\right)  \geq \sum_{j=\ceil*{n\tilde{x}_0-\alpha\sqrt n} }^{\floor*{n\tilde{x}_0+\alpha\sqrt n} }\mathbb{P}_j\left(\left\{\sigma\in \mathfrak{S}_j,\forall i \leq k,\frac{\lambda'_i(\sigma)-2\sqrt{n\tilde{x}_0}+\tilde{k}}{(n\tilde{x}_0)^\frac 16}\leq s_i\right\}\right)p^n_j.
\end{align}
Here, $\floor*{x}$ and $\ceil*{x}$ are respectively the floor and the ceiling functions. 
\paragraph*{} If $|j-n\tilde{x}_0|<\alpha\sqrt{n}$, then
\begin{align*}
\left|\sqrt{j}-\sqrt{n\tilde{x}_0} \right|\leq \frac{\alpha\sqrt{n}}{\sqrt{j}+\sqrt{n\tilde{x}_0}} \leq \frac{\alpha}{\sqrt{\tilde{x}_0}}.
\end{align*}
Thus,
\begin{align*}
\mathbb{P}_j\left(\left\{\sigma\in \mathfrak{S}_j,\forall i \leq k,\frac{\lambda'_i(\sigma)-2\sqrt{n\tilde{x}_0}+\tilde{k}}{(n\tilde{x}_0)^\frac 16}\leq s_i\right\}\right) \geq \mathbb{P}_j\left(\left\{\sigma \in \mathfrak{S}_j,\forall i \leq k,\frac{\lambda'_i(\sigma)-2\sqrt{j}}{j^\frac 16}\leq h(s_i,n)-\frac{2\alpha+\tilde{k}}{j^\frac 16}\right\}\right)
\end{align*}
and 
\begin{align*}
\mathbb{P}_j\left(\left\{\sigma\in \mathfrak{S}_j,\forall i \leq k,\frac{\lambda'_i(\sigma)-2\sqrt{n\tilde{x}_0}}{(n\tilde{x}_0)^\frac 16}\leq s_i\right\}\right) \leq \mathbb{P}_j\left(\left\{\sigma \in \mathfrak{S}_j,\forall i \leq k,\frac{\lambda'_i(\sigma)-2\sqrt{j}}{j^\frac 16}\leq -h(-s_i,n)+\frac{2\alpha}{j^\frac 16}\right\}\right).
\end{align*}
where, $h(s,n)=s\left({1-\frac{\alpha}{\sqrt n}}\right)^\frac 16$ if $s>0$ and $h(s,n)=s\left({1+\frac{\alpha}{\sqrt n}}\right)^\frac 16$ otherwise. 
\paragraph*{}
By the continuity and the monotony on each variable of $F_{2,k}$, there exists $\delta >0$ such that:
\begin{align*}
F_{2,k}(s_1,\dots,s_{k})-\varepsilon<F_{2,k}(s_1-\delta,\dots,s_{k}-\delta)<F_{2,k}(s_1+\delta,\dots,s_{k}+\delta)<F_{2,k}(s_1,\dots,s_{k})+\varepsilon.
\end{align*}
Moreover, there exists $n_1>n_0$ such that for all $n>n_1$, for all $j>n\tilde{x}_0-\alpha\sqrt{n}$, for all $i<k$,
\begin{equation*}
s_i-\delta \leq h(s_i,n)-\frac{2\alpha+\tilde{k}}{j^\frac 16}
\end{equation*}
and 
\begin{equation*}
s_i+\delta >-h(-s_i,n)+\frac{2\alpha}{j^\frac 16}.
\end{equation*}
Consequently, if $n>n_1$, inequalities \eqref{eq8} and \eqref{eq9} become respectively:
\begin{align} \label{eq10} 
\mathbb{P}\left(\forall i \leq k,\frac{\lambda'_i(\sigma_n)-2\sqrt{n\tilde{x}_0}}{(n\tilde{x}_0)^\frac 16}\leq s_i\right)   \leq \varepsilon+\sum_{j=\ceil*{n\tilde{x}_0-\alpha\sqrt n} }^{\floor*{n\tilde{x}_0+\alpha\sqrt n} }\mathbb{P}_j\left(\left\{\sigma\in \mathfrak{S}_j,\forall i \leq k,\frac{\lambda'_i(\sigma)-2\sqrt{j}}{j^\frac 16}\leq s_i+\delta\right\}\right)p^n_j
\end{align}
and
\begin{align}\label{eq11}
\mathbb{P}\left(\forall i \leq k,\frac{\lambda'_i(\sigma_n)-2\sqrt{n\tilde{x}_0}}{(n\tilde{x}_0)^\frac 16}\leq s_i\right)   \geq \sum_{j=\ceil*{n\tilde{x}_0-\alpha\sqrt n} }^{\floor*{n\tilde{x}_0+\alpha\sqrt n} }\mathbb{P}_j\left(\left\{\sigma\in \mathfrak{S}_j,\forall i \leq k,\frac{\lambda'_i(\sigma)-2\sqrt{j}}{j^\frac 16}\leq s_i-\delta\right\}\right)p^n_j.
\end{align}
Under \eqref{convrh}, $$\mathbb{P}_j\left(\left\{\sigma\in \mathfrak{S}_j,\forall i \leq k,\frac{\lambda'_i(\sigma)-2\sqrt{j}}{j^\frac 16}\leq s_i+\delta\right\}\right) \xrightarrow[j\to \infty]{} F_{2,k}(s_1+\delta,\dots,s_{k}+\delta),$$ and $$\mathbb{P}_j\left(\left\{\sigma\in \mathfrak{S}_j,\forall i \leq k,\frac{\lambda'_i(\sigma)-2\sqrt{j}}{j^\frac 16}\leq s_i-\delta\right\}\right) \xrightarrow[j\to \infty]{} F_{2,k}(s_1-\delta,\dots,s_{k}-\delta).$$
 Therefore, since $\ceil*{n\tilde{x}_0-\alpha\sqrt n} \to \infty$, there exists $n_2>n_1$ such that $\forall n>n_2$, $\forall j\geq \ceil*{n\tilde{x}_0-\alpha\sqrt n}$,
 \begin{align*}
 F_{2,k}(s_1-\delta,\dots,s_{k}-\delta)-\varepsilon &<\mathbb{P}_j\left(\left\{\sigma\in \mathfrak{S}_j,\forall i \leq k,\frac{\lambda'_i(\sigma)-2\sqrt{j}}{j^\frac 16}\leq s_i-\delta\right\}\right)\\&<\mathbb{P}_j\left(\left\{\sigma\in \mathfrak{S}_j,\forall i \leq k,\frac{\lambda'_i(\sigma)-2\sqrt{j}}{j^\frac 16}\leq s_i+\delta\right\}\right)< F_{2,k}(s_1-\delta,\dots,s_{k}-\delta)+\varepsilon .
 \end{align*}
 Finally, if $n>n_2$, using  \eqref{eq7},  inequalities  \eqref{eq10} and \eqref{eq11} become
 \begin{align*}
(F_{2,k}(s_1,\dots,s_{k})-2\varepsilon)(1-\varepsilon) < \mathbb{P}\left(\forall i \leq k,\frac{\lambda'_i(\sigma_n)-2\sqrt{n\tilde{x}_0}}{(n\tilde{x}_0)^\frac 16}\leq s_i\right) <F_{2,k}(s_1,\dots,s_{k})+3\varepsilon,
 \end{align*}
\end{proof}
and the proof of the proposition is therefore complete.
\subsection{Proof of results for the descent process}
\label{descpr}
\paragraph*{} In this subsection, we prove the convergence of the descent process  for  some random permutations stable under conjugation (Theorem \ref{thm2}). We prove also results of convergence for virtual permutations  (Theorem \ref{5}, Corollary \ref{gcase} and Proposition \ref{1.1}).
\paragraph*{} 
Let $A$ be a finite subset of $\mathbb{N}^*$ and $m:=\max(A)$ and let $A'=\{1,2,\dots,m+1\}$. The idea of the proof of Theorem \ref{thm2} is to study the descent process under the condition $\{\sigma_n (A')\cap A'=\emptyset\}$ and to show that it does not depend on the law of $\sigma_n$. 
\begin{lemma}
Let  $E_n:=\{\sigma \in \mathfrak{S}_n,\sigma (A')\cap A'=\emptyset\}$. Assume that the law of $\sigma_n$ is stable under conjugation and $\mathbb{P}(\sigma_n\in E_n)>0$.Then  for any $b_1,b_2,\dots,b_{m+1}$ distinct elements of $\{1,\dots,n\}$, 
\begin{align*}
\mathbb{P}(\sigma_n(1)=b_1,\dots,\sigma_n(m+1)=b_{m+1} |E_n)=\frac{\mathbbm{1}_{\min_{i}(b_i)>m+1}}{\binom{n-m-1}{m+1}}.
\end{align*}
\label{lemmadesc}
\end{lemma}
\begin{proof}
 The event $E_n$ reads as the disjoint union of the events $\{\sigma(1)=b_1,\dots,\sigma(m+1)=b_{m+1}\}$ where 
$b_1,b_2,\dots,b_{m+1}$ are distinct elements of $\{m+2,m+3,\dots,n\}$.
Let $b_1,b_2,\dots,b_{m+1}$ and $c_1,c_2,\dots,c_{m+1}$ verify the previous condition. Let $\hat\sigma \in \mathfrak{S}_n$  be a permutation such that   for any $1\leq i \leq m+1, \hat\sigma(c_i)=b_i$ and $\hat\sigma(j)=j$ if  $j \notin (\{b_i\}_{i\leq m+1}\cup\{c_i\}_{i\leq m+1})$.  By invariance under conjugation, we have
\begin{align*}
\mathbb{P}(\sigma_n(1)=b_1,\dots,\sigma_n(m+1)=b_{m+1})&=
\mathbb{P}(\hat\sigma\circ\sigma_n\circ \hat\sigma^{-1}(1)=b_1,\dots,\hat\sigma\circ\sigma_n\circ \hat\sigma^{-1}(m+1)=b_{m+1})\\&=\mathbb{P}(\sigma_n(1)=c_1,\dots,\sigma_n(m+1)=c_{m+1})
\end{align*}
and thus\begin{align*}
\mathbb{P}(\sigma_n(1)=b_1,\dots,\sigma_n(m+1)=b_{m+1} |E_n)&=
\mathbb{P}(\sigma_n(1)=c_1,\dots,\sigma_n(m+1)=c_{m+1}|E_n)
\end{align*}
and the lemma follows. 
\end{proof}
\begin{proof}[Proof of Theorem \ref{thm2}]
Under \eqref{h3},
\begin{align*}
\mathbb{P}(\sigma_n \in E_n)\geq 1-\sum_{i=1}^{m+1} \mathbb{P}(\sigma_n(i)\leq m+1) =  1-(m+1)
\left(\mathbb{P}(\sigma_n(1)=1)+\frac{m(1-\mathbb{P}(\sigma_n(1)=1))}{n-1}\right) \xrightarrow[n\to \infty]{} 1.
\end{align*}
Similarly, if $\tilde{\sigma}_n$ follows the uniform distribution on $\mathfrak{S}_n$, we have
$\mathbb{P}(\tilde{\sigma}_n \in {E_n})\to 1.$ Therefore, since the law of $\sigma_n$ is invariant under conjugation  \eqref{h1} we can use  Lemma \ref{lemmadesc}  for $n$ large enough to get 
\begin{equation*} 
 \mathbb{P}(A \subset D(\sigma_n)|E_n)= \mathbb{P}(A \subset D(\tilde{\sigma}_n)|E_n).
\end{equation*}
Thus,
 \begin{equation*} 
\lim_{n\to \infty} \left(\mathbb{P}(A \subset D(\sigma_n))- \mathbb{P}(A \subset D(\tilde{\sigma}_n))\right)=0.
\end{equation*}
Since   $\tilde{\sigma}_n$ satisfies \eqref{main} by Theorem \ref{borodin2}, this concludes the proof.
\end{proof}
\paragraph*{} Before proving  Theorem \ref{5}, we need to recall that a point process $X$ on a discrete  space $\mathfrak{X}$ is fully characterised by its correlation function (we denote it by  $\rho$). Given $A$ a finite subset of $\mathfrak{X}$, 
 \begin{align*}
 \rho(A):=\mathbb{P}(A \subset X). 
 \end{align*}
It is called determinantal with kernel $K$ if for all $A$ finite subset of $\mathfrak{X}$, 
\begin{align}
\rho(A)=\det\left([K(i,j)]_{i,j\in A}\right).
\end{align}
A point process defined on $\mathbb{N}^*$  is  1-dependent if  for all $A$ and $B$ finite subsets of  $\mathbb{N}^*$ such that the distance between A and B is larger than 1, $\rho(A \cap B)=\rho(A) \rho(B)$.  It is called stationary on  $\mathbb{N}^*$ if for all  positive integer $k$, for all finite subset  $A \subset{\mathbb{N}^*}$,
$\rho(A)=\rho(A+k).$
\paragraph*{}
To prove Theorem \ref{5}, we will use the following result. 
\begin{theorem} \citep*{MR2721041} \label{bordin} A stationary 1-dependent simple point process on $\mathbb{N}^*$ is determinantal  with  kernel $K$ given by $K(i,j)=k(j-i)$ and 
\begin{equation*}
\sum_{i \in \mathbb{Z}} k(i)z^i=\frac{-1}{z+\sum_{i\geq 1} a_iz^{i+1}},
\end{equation*}
where $a_i:=\rho(\{1,2,\dots,i\})$.
\end{theorem}
\paragraph*{}
\begin{proof}[Proof of Theorem \ref{5}] 
If $x_0=1$, the theorem is obvious since $D(\sigma^\nu_n)=\delta_{\emptyset}$. Next we split the proof into two steps depending on  whether $x_0=0$ or not.
\paragraph*{}
Step 1 : We assume $x_0=0$ so that $\nu(\Sigma_1)=1$.  
Using equalities \eqref{deferg} and \eqref{onem} we obtain:
\begin{align*}
\mathbb{P}(\sigma^\nu_n(1)=1)=\sum_{\sigma\in \mathfrak{S}_n, \sigma(1)=1} \mathbb{P}(\sigma^\nu_n=\sigma)&=\sum_{\sigma\in \mathfrak{S}_n, \sigma(1)=1} \int_{\Sigma_1}f(n,x,\sigma)d\nu(x)\\&= \int_{\Sigma_1}\sum_{\sigma\in \mathfrak{S}_n, \sigma(1)=1}f(n,x,\sigma)d\nu(x)\\&=\int_{\Sigma_1}\mathbb{P}(\sigma^{\delta_x}_n(1)=1)d\nu(x).
\end{align*}
Using Beppo Levi theorem, it is thus enough to prove
\begin{align*}
\mathbb{P}(\sigma^{\delta_x}_n(1)=1)\to 0.
\end{align*}
Using the same combinatorial interpretation as in the beginning of Subsection \ref{appli}, we have for any $x\in \Sigma_1$,
\begin{align*}
\mathbb{P}(\sigma^{\delta_x}_n(1)=1)
=\sum_{i\geq 1}\mathbb{P}(\sigma^{\delta_x}_n(1)=1|pos_1=i)\mathbb{P}(pos_1=i)= \sum_{i\geq 1} x_i(1-x_i)^{n-1}.
\end{align*}
Let $\varepsilon>0$. Since $\sum_i{x_i}=1$, there exists $n_0$ such that $(\sum_{i>n_0}x_i)<\frac{\varepsilon}{2}$ and 
\begin{align*}
\mathbb{P}(\sigma^{\delta_x}_n(1)=1)
= \sum_{i\geq 1} x_i(1-x_i)^{n-1}\leq\sum_{i= 1}^{n_0} x_i(1-x_i)^{n-1} +\frac{\varepsilon}{2}.
\end{align*}
As for all $i\leq n_0$, $x_i(1-x_i)^{n-1}$ converges to $0$ when $n$ goes to infinity, there exists $n_1$ such that for $n>n_1$
$\sum_{i= 1}^{n_0} x_i(1-x_i)^{n-1}<\frac{\varepsilon}{2}$ and therefore
\begin{align*}
\mathbb{P}(\sigma^{\delta_x}_n(1)=1)\to 0.
\end{align*}
Theorem \ref{5} follows from Theorem \ref{thm2} when $x_0=0$.

\paragraph*{} 
Step 2: we now assume that $0<x_0<1$ and  $\nu(\Sigma_{1-x_0})=1$. We have
\begin{align*}
\mathbb{P}(\sigma^\nu_n(1)=1)=x_0+\int_{\Sigma}\sum_{i\geq1} x_i(1-x_i)^{n-1} d\nu(x)\geq x_0>0,
\end{align*}
which prevents the use of Theorem \ref{thm2}. The strategy  is instead to use 
Theorem \ref{bordin}, namely to prove  that the limiting process is stationary, 1-dependent and its correlation function is such that  $\forall k \geq 1$, $$\rho(\{1,2,\dots,k\})=\frac{(1-x_0)^{k+1}}{(k+1)!}+\frac{x_0(1-x_0)^k}{k!}.$$
To do so we need to prove this result in the particular case  $d\nu_1(x):=dPD(1)(\frac{x}{1-x_0})$  
 since for any finite subset B,
\begin{equation}\label{eqfin} \lim_{n\to \infty}\left(\mathbb{P}(B\subset D\left(\sigma^\nu_n)\right)-{\mathbb{P}}(B\subset D(\sigma^{{\nu_1}}_n))\right)=0. \end{equation}
Indeed, let $B$ be a finite subset of $\mathbb{N}^*$ and  $B':=B\cup (B+1)$.
We use the same interpretation of the random virtual permutations in this case as in the proof of Proposition~\ref{them11}.  We choose a random subset $A_n$ of $\{1,2,\dots,n\}$ of fixed points where each point belongs to $A_n$  with probability $x_0$ independently from the others. After that, we permute the elements according to  $\mathbb{P}_{n-|A_n|}$, where $(\mathbb{P}_n)_{n\geq1}$ is the probability distribution on $\mathfrak{S}^\infty$ associated to $\hat\nu$ where $d\hat\nu(x)=d{\nu}(\frac{x}{1-x_0})$. Let $C_n:=A_n\cap B'$ and  $$E_n:=\{\sigma \in \mathfrak{S}_n,\forall i\in B'\setminus{C_n},\, \sigma(i)>\max(B') \}.$$
We have
\begin{equation*}
\mathbb{P}\left(B\subset D\left(\sigma^\nu_n\right)\middle|E_n\right)=\sum_{X \subset B'} \mathbb{P}\left(B\subset D\left(\sigma^\nu_n\right)\middle| E_n,C_n=X\right)\mathbb{P}(C_n=X).
\end{equation*}
With similar arguments as in the proof of Lemma \ref{lemmadesc}, it is not difficult to show that the quantity $\mathbb{P}(B\subset D\left(\sigma^\nu_n\right)| E_n,C_n=X)$  is defined for $n>|B'| +\max(B')$  and does not depend on $\nu$. Moreover, $\mathbb{P}(C_n=X)=x_0^{|X|}(1-x_0)^{|B'|-|X|}$. Thus $\mathbb{P}\left(B\subset D\left(\sigma^\nu_n\right)\middle|E_n\right)$ does not depend on $\nu$.  We have
\begin{align*}
\mathbb{P}(\sigma_n^\nu\in E_n)&=\sum_{X \subset B'} \mathbb{P}(\sigma_n^\nu\in E_n| C_n=X)\mathbb{P}(C_n=X) \\&\geq 
1-\sum_{X \subset B'}\sum_{j \in B'\setminus X}\mathbb{P}(\sigma^\nu_n(j)\leq\max(B')|C_n=X) \mathbb{P}(C_n=X).
\end{align*}
Moreover, using the notation  $p_k:=\mathbb{P}(\sigma_k^{\hat\nu}(1)=1)$ and observing that $p_k\to0$ as $k\to\infty$ thanks to Step 1, we have
\begin{align*}
&\mathbb{P}(\sigma^\nu_n(j)\leq\max(B')|C_n=X) \\=&\sum_{k=0}^{n-|B'|} \mathbb{P}(\sigma^\nu_n(j)\leq\max(B')|C_n=X,|A_n|=|X|+n-|B'|-k)\mathbb{P}(|A_n|=|X|+n-|B'|-k|C_n=X)
\\=&\sum_{k=0}^{n-|B'|} x_0^{n-|B'|-k}(1-x_0)^k\binom{n-|B'|}{k} \mathbb{P}(\sigma^\nu_n(j)\leq\max(B')|C_n=X,|A_n|=|X|+n-|B'|-k)
\\\leq &x_0^{n-|B'|}+x_0^{n{-|B'|}-1}(1-x_0)(n-|B'|)+\sum_{k=2}^{n-|B'|} x_0^{n-|B'|-k}(1-x_0)^k\binom{n-|B'|}{k}\left({p}_{k+|B'|-|X|}+\frac{\max(B')}{|B'|-|X|+k-1}\right) \\&\xrightarrow[n\to \infty]{}0. 
\end{align*}
This yields
\begin{align*}
\lim_{n\to\infty}\mathbb{P}(\sigma_n^\nu\in E_n)=1
\end{align*}
and therefore the claim \eqref{eqfin} is proven.
\paragraph*{} We compute now $$\lim_{n\to \infty}{\mathbb{P}}(B\subset D(\sigma^{{\nu_1}}_n)).$$
The finite subset $B$ can be decomposed as  $B=\bigcup_{i=1}^l B_i$  where each $B_i$ consists in  consecutive elements of $\mathbb{N}^*$ and the distance between $B_i$ and $B_j$ is larger than one if $i\neq j$. 
For example,$$B=\{1,2,3,5,6,8,11,12\}=\{1,2,3\}\cup\{5,6\}\cup\{8\}\cup\{11,12\}.$$
Note that every finite subset has a such decomposition. Let $B'_i:=B_i\cup(B_i+1)$.
We have  $B':=B\cup(B+1)=\bigcup_{i=1}^l B'_i$    and if  $i\neq j$, then $B'_i \cap B'_j=\emptyset$.  
From now we assume that  $n>|B'|+\max(B')$. We have 
\begin{equation}\label{eq21}
{\mathbb{P}}(B\subset D(\sigma^{\nu_1}_n)| E_n)=\sum_{X \subset B'}{\mathbb{P}}(B\subset D(\sigma^{\nu_1}_n)| C_n=X,E_n) {\mathbb{P}}(C_n=X).
\end{equation}
If  $B\cap X\neq\emptyset$, then  
${\mathbb{P}}(B\subset D(\sigma^{\nu_1}_n)| C_n=X,E_n)=0$. Indeed, conditionally on  $E_n$, if $i \in B\cap X$,  then   $\sigma^{\nu_1}_n(i)=i$ and $\sigma^{\nu_1}_n(i+1)$ is either equal to  $i+1$ or larger than $\max(B')$ and in both cases, there is no descent on $i$. Consequently,  \eqref{eq21} becomes
\begin{align*}
{\mathbb{P}}(B\subset D(\sigma^{\nu_1}_n)| E_n)&=\sum_{X \subset B'\setminus B}{\mathbb{P}}(B\subset D(\sigma^{\nu_1}_n)| C_n=X,E_n) {\mathbb{P}}(C_n=X)\\&= \sum_{U \subset \{1,2,\dots,l\}}{\mathbb{P}}\left(B\subset D(\sigma^{\nu_1}_n)\middle| C_n=\bigcup_{i\in U} (B'_i \setminus B_i),E_n\right) {\mathbb{P}}\left(C_n=\bigcup_{i\in U} (B'_i \setminus B_i),E_n\right).
\end{align*}
The second equality comes from the fact that $B'_i\setminus B_i$ contains exactly one element. We denote by $U^c:=\{1,2,\dots,l\}\setminus U$ and by $W(U):=\bigcup(\bigcup_{i\in U}B_i \bigcup_{i \in U^c}B'_i) $. We have
\begin{align*}
{\mathbb{P}}\left(B\subset D(\sigma^{\nu_1}_n)\middle| C_n=\bigcup_{i\in U} (B'_i \setminus B_i),E_n\right)= \frac{ |\mathfrak{E}_2|}{|\mathfrak{E}_1|},
\end{align*}
where 
\begin{align*}
\mathfrak{E}_1&=\left\{(e_k)_{k\in W(U) },\forall k \in W(U),\  \max(B')<e_k\leq n, i\neq j \Rightarrow e_i\neq e_j \right\}
\end{align*}
and 
\begin{align*}
\mathfrak{E}_2&:=\left\{(e_k)_{k\in W(U) } \in \mathfrak{E}_1 ,\forall k \in \bigcup_{i=1}^l B_i\setminus \bigcup_{i\in U}\{\max (B_i)\}, \  e_{k+1}<e_{k} \right\}.
\end{align*}
Therefore,
\begin{align*}
|\mathfrak{E}_1|:=\frac{(n-\max(B'))!}{(n-\max (B)'-|W(U)|)!},
\end{align*}
and
\begin{align*}
|\mathfrak{E}_2| &= \frac{(n-\max(B'))!}{(n-\max(B')-\sum_{i\in U}|B_i|)!\prod_{i\in U} |B_i|! } \frac{(n-\max(B')-\sum_{i\in U}|B_i|)!}{(n-\max(B')-\sum_{i\in U}|B_i|-\sum_{i\in U^c}|B'_i|)!\prod_{i\in U^c} |B'_i|!}
\\ &=\frac{(n-\max(B'))!}{(n-\max(B')-|W(U)|)!\prod_{i\in U} |B_i|!\prod_{i\in U^c} |B'_i|!}.
\end{align*}
As a consequence, 
\begin{align*}
{\mathbb{P}}\left(B\subset D(\sigma^{\nu_1}_n)\middle| C_n=\bigcup_{i\in U} (B'_i \setminus B_i),E_n\right)= \frac{ |\mathfrak{E}_2|}{|\mathfrak{E}_1|}=\frac{1}{\prod_{i\in U} |B_i|!\prod_{i\in U^c} |B'_i|!}.
\end{align*}
Then
\begin{align*}
{\mathbb{P}}(B\subset D(\sigma^{\nu_1}_n)| E_n)= \sum_{U \subset \{1,2,\dots,l\}} 
\frac{ x_0^{|U|}(1-x_0)^{|B|+l-|U|}}{\prod_{i\in U} |B_i|!\prod_{i\in U^c} |B'_i|!}&=\prod_{i=1}^l \frac{(1-x_0)^{|B_i|}}{|B_i|!} \left(x_0+\frac{1-x_0}{|B_i|+1}\right)\\&=\prod_{i=1}^l \hat{a}_{|B_i|}(x_0),
\end{align*}
where we recall that  
\begin{align*}
\hat{a}_k(x_0)=\frac{(1-x_0)^{k+1}}{(k+1)!}+\frac{x_0(1-x_0)^{k}}{k!}.
\end{align*}
This implies that the limiting process is stationary and 1-dependent. Consequently by Theorem \ref{bordin}   it is determinantal and   the kernel satisfies \eqref{keneldescvirt}.
\end{proof}

\paragraph*{} Corollary \ref{gcase} is at the same time a generalization and a direct application of  Theorem \ref{5}.
\begin{proof}[Proof of Corollary \ref{gcase}]
 We denote by $f(n,x,\sigma):= \mathbb{P}\left(\sigma^{\delta_x}_n=\sigma\right)$  (see \eqref{defdef}), by $\rho(n,x,.)$ the correlation function of the descent process of $\sigma^{\delta_x}_n$ and by $\rho_{lim}(x_0,.)$ the correlation function of the determinantal process with kernel ${K_{x_0}(i,j):=k_{x_0}(j-i)}$. Let $A$ be a finite subset of $\mathbb{N}^*$. We have
 \begin{align*}
 \mathbb{P}(A\subset D(\sigma^\nu_n))=\sum_{\sigma \in \mathfrak{S}_n,A\subset D(\sigma)} \mathbb{P}(\sigma^\nu_n=\sigma)&=\sum_{\sigma \in \mathfrak{S}_n,A\subset D(\sigma)} \int_\Sigma f(n,x,\sigma)d\nu(x)\\&= \int_\Sigma \sum_{\sigma \in \mathfrak{S}_n,A\subset D(\sigma)} f(n,x,\sigma)d\nu(x)
\\&=\int_\Sigma \rho(n,x,A)d\nu(x). \end{align*}
Using the convergence of $\rho(n,x,A)$ to $\rho_{lim}(1-\sum_{i\geq 1}x_i,A)$ and the dominated convergence theorem, we obtain:
\begin{align*}
  \mathbb{P}(A\subset D(\sigma^\nu_n)) \xrightarrow[n\to\infty]{} \int_{\Sigma}\rho_{lim}\left(1-\sum_{i\geq 1}x_i,A\right) d\nu(x).
 \end{align*}
 \end{proof}
Using this corollary, we can now proove Proposition \ref{1.1}.
\begin{lemma}\label{finallemma}
For any  random permutation $\sigma_n$ stable under conjugation, $\mathbb{P}(i \in D(\sigma_n))$ does not depend on $i$.
\end{lemma}
\begin{proof}
Let $1\leq i<n$. We have 
\begin{align*}
\mathbb{P}(i \in D(\sigma_n))&=
\mathbb{P}(i \in D(\sigma_n)|\sigma_n(i)=i,\sigma_n(i+1)=i+1)\mathbb{P}(\sigma_n(i)=i,\sigma_n(i+1)=i+1)
\\&+\mathbb{P}(i \in D(\sigma_n)|\sigma_n(i)=i,\sigma_n(i+1)\neq i+1)\mathbb{P}(\sigma_n(i)=i,\sigma_n(i+1)\neq i+1)
\\&+
\mathbb{P}(i \in D(\sigma_n)|\sigma_n(i)\neq i,\sigma_n(i+1)=i+1)\mathbb{P}(\sigma_n(i)\neq i,\sigma_n(i+1)= i+1)
\\&+\mathbb{P}(i \in D(\sigma_n)|\sigma_n(i)\notin \{i,i+1\},\sigma_n(i+1)\notin \{i,i+1\})\mathbb{P}(\sigma_n(i)\notin \{i,i+1\},\sigma_n(i+1)\notin \{i,i+1\})
\\&+\mathbb{P}(i \in D(\sigma_n)| \sigma_n(i)=i+1,\sigma_n(i+1)\notin \{i,i+1\})\mathbb{P}( \sigma_n(i)=i+1,\sigma_n(i+1)\notin \{i,i+1\})
\\&+\mathbb{P}(i \in D(\sigma_n)|\sigma_n(i)\notin \{i,i+1\},\sigma_n(i+1)=i)\mathbb{P}(\sigma_n(i)\notin \{i,i+1\},\sigma_n(i+1)=i)
\\&+\mathbb{P}(i \in D(\sigma_n)|\sigma_n(i)=i+1,\sigma_n(i+1)=i)\mathbb{P}(\sigma_n(i)=i+1,\sigma_n(i+1)=i)
.
\end{align*}
Using the stability under conjugation, we obtain,
\begin{align*}
\mathbb{P}(i\in D(\sigma_n)|\sigma_n(i)=i,\sigma_n(i+1)=i+1)&=0
\\ \mathbb{P}(i \in D(\sigma_n)|\sigma_n(i)=i,\sigma_n(i+1)\neq i+1)&=\frac{i-1}{n-2}
\\\mathbb{P}(i \in D(\sigma_n)|\sigma_n(i)\neq i,\sigma_n(i+1)=i+1)&=\frac{n-i-1}{n-2}
\\\mathbb{P}(i \in D(\sigma_n)|\sigma_n(i)\notin \{i,i+1\},\sigma_n(i+1)\notin \{i,i+1\})&=\frac{1}{2}
\\\mathbb{P}(i \in D(\sigma_n)| \sigma_n(i)=i+1,\sigma_n(i+1)\notin \{i,i+1\})&=\frac{i-1}{n-2}
\\ \mathbb{P}(i \in D(\sigma_n)|\sigma_n(i)\notin \{i,i+1\},\sigma_n(i+1)=i)&=\frac{n-i-1}{n-2}
\\\mathbb{P}(i\in D(\sigma_n)|\sigma_n(i)=i+1,\sigma_n(i+1)=i)&=1.
\end{align*}
We have then, using again the stability under conjugation,
\begin{align*}
\mathbb{P}(i \in D(\sigma_n))&=\mathbb{P}(\sigma_n(1)=1,\sigma_n(2)\neq 2)\\&+\mathbb{P}(\sigma_n(1)=2,\sigma_n(2)=1)\\&+\mathbb{P}(\sigma_n(1)\notin \{1,2\},\sigma_n(2)=1)\\&+\frac{1}{2}\mathbb{P}(\sigma_n(1)\notin \{1,2\},\sigma_n(2)\notin \{1,2\})
\end{align*}
and the lemma follows.
\end{proof}
\begin{proof}[Proof of Proposition \ref{1.1}]
Let $\nu$ be a probability measure on $\Sigma$. By Lemma \ref {finallemma} and by using \eqref{eqq7} and \eqref{limvir} for $A=\{1\}$, we obtain
\begin{align*}
\frac{\mathbb{E}(|D(\sigma^\nu_n)|)}{n}=\frac{n-1}{n}\mathbb{P}\left(1 \in D(\sigma^\nu_n)\right)\to\int_\Sigma \hat{a}_1\left(1-\sum_{i\geq1}{x_i}\right)d\nu(x)=\frac{1}{2}\left(1-\int_\Sigma \left(1-\sum_i{x_i}\right)^2 d\nu(x)\right).
\end{align*}
\end{proof}

\end{document}

%% file: bibmath.tex
 \usepackage{amsmath}
 \usepackage{mathtools}
\usepackage{amsfonts}
\usepackage{amssymb}
\usepackage{graphicx}
\usepackage{amsthm}
\usepackage{bbm}
\usepackage{mathtools}
\usepackage[colorlinks=true,
            linkcolor=blue,
            urlcolor=blue,
            citecolor=blue]{hyperref}
\usepackage{authblk}

\usepackage{geometry}
\usepackage{algorithm}
\usepackage[noend]{algpseudocode}
\usepackage[vcentermath]{youngtab}
\usepackage{pgfplots}
\pgfplotsset{compat=newest}
\usetikzlibrary{calc}
\usepackage{mathrsfs}
\DeclarePairedDelimiter\ceil{\lceil}{\rceil}
\DeclarePairedDelimiter\floor{\lfloor}{\rfloor}

%% file: diag.pdf_tex
\begingroup%
  \makeatletter%
  \providecommand\color[2][]{%
    \errmessage{(Inkscape) Color is used for the text in Inkscape, but the package 'color.sty' is not loaded}%
    \renewcommand\color[2][]{}%
  }%
  \providecommand\transparent[1]{%
    \errmessage{(Inkscape) Transparency is used (non-zero) for the text in Inkscape, but the package 'transparent.sty' is not loaded}%
    \renewcommand\transparent[1]{}%
  }%
  \providecommand\rotatebox[2]{#2}%
  \newcommand*\fsize{\dimexpr\f@size pt\relax}%
  \newcommand*\lineheight[1]{\fontsize{\fsize}{#1\fsize}\selectfont}%
  \ifx\svgwidth\undefined%
    \setlength{\unitlength}{432bp}%
    \ifx\svgscale\undefined%
      \relax%
    \else%
      \setlength{\unitlength}{\unitlength * \real{\svgscale}}%
    \fi%
  \else%
    \setlength{\unitlength}{\svgwidth}%
  \fi%
  \global\let\svgwidth\undefined%
  \global\let\svgscale\undefined%
  \makeatother%
  \begin{picture}(1,0.66666667)%
    \lineheight{1}%
    \setlength\tabcolsep{0pt}%
    \put(0,0){\includegraphics[width=\unitlength,page=1]{diag.pdf}}%
  \end{picture}%
\endgroup%